\documentclass[11pt]{article}
\usepackage{epsfig,amsmath,amssymb,latexsym,color,hyperref,subcaption}

\newcounter{mnote}
  \let\oldmarginpar\marginpar
 \renewcommand\marginpar[1]{\-\oldmarginpar[\raggedleft\footnotesize #1]%
    {\raggedright\footnotesize #1}}

\def\bfa{{\bf a}}

\def\bfc{{\bf c}}
\def\bfd{{\bf d}}

\def\bff{{\bf f}}
\def\bfg{{\bf g}}
\def\bfi{{\bf i}}

\def\bfx{{\bf x}}

\def\bfz{{\bf z}}

\def\PML{{\rm PML}}

\def\3bar{{|\hspace{-.02in}|\hspace{-.02in}|}}
\newtheorem{theorem}{Theorem}
\newtheorem{lemma}{Lemma}

\newtheorem{definition}{Definition}
\newenvironment{proof}{\begin{trivlist}\item[]{\emph{Proof.}}}
               {\hfill$\Box$\end{trivlist}}

\begin{document}
\title{A Bivariate Spline Solution to the Exterior Helmholtz Equation \\and its Applications}
\author{Shelvean Kapita\footnote{shelvean.kapita@uga.edu, Department of Mathematics,
University of Georgia, Athens, GA 30602.}
\and Ming-Jun Lai
\footnote{mjlai@uga.edu. Department of Mathematics,
University of Georgia, Athens, GA 30602. 
This author is partially supported by 
the National Science Foundation under
the grant \#DMS 1521537. } }
\maketitle

\begin{abstract}
We explain how to use smooth bivariate splines of arbitrary degree to solve the exterior Helmholtz equation 
based on a Perfectly Matched Layer (PML) technique. In a previous study (cf. \cite{LM18}), 
it was shown that  bivariate spline functions of high degree can  
approximate the solution of the bounded domain Helmholtz equation with an impedance boundary condition for large wave numbers 
$k\sim 1000$. In this paper, we extend this study to the case of the Helmholtz equation in an unbounded domain. The PML is 
constructed using a complex stretching of the coordinates in a rectangular domain, resulting in a weighted Helmholtz equation with 
Dirichlet boundary conditions. The PML weights are also approximated by using spline functions. The computational algorithm
developed in \cite{LM18} is used to approximate the solution of the resulting weighted Helmholtz equation.  
Numerical results show that the PML 
formulation for the unbounded domain Helmholtz equation using bivariate splines is very effective over a range of examples.
\end{abstract}

\section{Introduction}
We are interested in approximating the numerical solution of the Helmholtz equation over an exterior domain with 
the Sommerfeld radiation condition at infinity. The problem can be described as follows: Find the scattered field $u$ that satisfies
\begin{eqnarray}
&& -\Delta u-k^2 u = f\;\;\mbox{in}\;\;\mathbb{R}^2\setminus\overline{D},\nonumber\\
&& u=g\;\;\mbox{on}\;\;\Gamma:=\partial D, \nonumber\\
&&\lim_{r\rightarrow\infty}\sqrt{r}\left(\dfrac{\partial u}{\partial r} -iku\right)=0,\;\;r=|\bfx|,\;\bfx\in\mathbb{R}^2
\label{exteriorP}
\end{eqnarray}
where $D$ is a polygonal domain, $f$ is an $L^2$ integrable function with compact support over $\mathbb{R}^2\setminus\overline{D}$, 
$u$ is the solution to be solved and $k: =\omega/c$ is the wave number, $\omega$ 
being the angular frequency of the waves and $c$ the sound speed or speed of light dependent 
on the applications. The last equation, the Sommerfeld radiation condition, is assumed to hold uniformly in all directions. We 
assume that the wavenumber  $k>0$ is constant. The more general case of inhomogeneous and anisotropic media will be considered in 
a future publication. In this paper, we assume a Dirichlet boundary condition on $D$, but the analysis and numerical results can 
be extended easily to Neumann or Robin boundary conditions. 

The Helmholtz equation arises in many applications as a simplified model of wave propagation in the frequency domain. The modeling 
of acoustic and electromagnetic wave propagation in the time-harmonic regime are well-known examples. In particular, there is 
increasing interest in understanding electromagnetic scattering by subwavelength apertures or holes \cite{LZ18, LZ18a, LZ17}. Scattering problems that 
involve periodic and layered media have also attracted attention \cite{DN12, HNS12}. As a result, robust and accurate numerical algorithms for 
solving the Helmholtz equation have been studied over the years. For low frequencies, these problems can be handled using low 
order finite elements and finite differences. However as the wavenumber is increased, low order finite elements and finite 
differences become very expensive due to small mesh sizes that are needed to adequately resolve the wave. High order finite 
elements, least squares finite elements, Trefftz Discontinuous Galerkin methods have been proposed to handle large wavenumbers. In 
the case of wave propagation in unbounded domains, accurate treatment of truncation boundaries is also of paramount importance. 

We plan to use bivariate spline functions to numerically solve the exterior Helmholtz problem [\ref{exteriorP}]. Since our method involves a volumetric discretization, we need to truncate the unbounded domain while minimizing reflections from the artificial boundary. In an ideal framework, this truncation should satisfy at least three 
properties: efficiency, easiness of implementation, and robustness. There are several approaches available
in the literature: infinite element methods (IFM)\cite{SB98}, boundary element methods (BEM) \cite{CZ92, H94}, Dirichlet-to-Neumann operators (DtN) \cite{DG04,NN06, HNPX11}, 
Absorbing Boundary Conditions (ABC) \cite{ZT13}, and Perfectly Matched Layer (PML) methods (e.g. \cite{B94} and many current
studies). The BEM approach involves a discretization of the partial differential equation on a lower dimensional surface. In the DtN and ABC approaches, a Robin \emph{boundary condition} is placed on the truncating boundary to absorb waves from the interior computational domain and to minimize reflections back into the domain. In PML approaches, a \emph{boundary layer} of finite thickness is placed on the exterior of the computational domain to absorb and attenuate any waves coming from the interior. Moreover, the PML layer should preserve the solution inside the computational region of interest while exponentially damping the wave inside the PML region. Since the wave becomes exponentially small inside the PML layer, a hard Dirichlet boundary condition can be imposed on the exterior boundary of the PML layer.

The PML technique was first introduced in [Berenger, 1994\cite{B94}] in the context of Maxwell's equations, 
but has now  become an efficient approach for dealing with exterior domain wave propagation problems such as 
Helmholtz equations and Navier equations in two and three dimensions. Many PML functions have been designed and implemented.  
The original PML by Berenger was constructed in Cartesian coordinates for the Maxwell equations, and extended to curvilinear 
coordinates by  Collino and Monk \cite{MC98}. In \cite{CW94}, Chew and Wheedon demonstrated that the PML can be constructed 
by a complex change of variables.

In this paper, we shall use the  PML technique in \cite{KP10a} and apply spline functions to solve 
(\ref{exteriorP}). Bivariate
and trivariate splines can be used for numerical solution of partial differential equations. We refer the 
reader to \cite{ALW06}, \cite{LW04}, \cite{HHL07}, \cite{GLS15}, \cite{LW18}, \cite{LM18}, and etc..  In 
particular, in \cite{LM18}, the Helmholtz equation over a bounded domain in $\mathbb{R}^2$ is solved 
numerically by using bivariate spline functions. The researchers in \cite{LM18} designed a bivariate spline method 
which very effectively solves the Helmholtz equation with 
large wave number with $k\ge 1000$.  This paper is a complement of \cite{LM18}.  

The rest of this paper is presented as follows. We first describe the PML formulation as a complex stretching of coordinates leading to a weighted Helmholtz equation. The exterior domain Helmholtz equation is then approximated by the PML problem with Dirichlet boundary conditions in a bounded domain. We apply the theory in \cite{LM18} to show the well-posedness of the truncated PML problem. We next present bivariate splines and prove error estimates of the approximation of splines to the solution of the truncated PML problem. We finish with several numerical results that demonstrate the convergence of spline functions to the PML problem.

\section{The PML formulation}
We shall use a Cartesian PML for the scattering problem  in the unbounded exterior of a 2D domain. We give a brief description 
of the setup, which can also be found in \cite{KP10a}. We shall assume $\Omega\subset\mathbb{R}^2$ is a bounded polygonal domain 
contained inside a rectangle $[-a_1,a_1]\times[-a_2,a_2]$ for some $a_1, a_2>0$.  The Cartesian PML is formulated in terms of the 
symmetric functions $\gamma_1(x_1)$ and $\gamma_2(x_2)$  for each point $\bfx = (x_1,x_2)$ in 
$\mathbb{R}^2\backslash\overline{\Omega}$. For some $b_j>0$ such that $0<a_j<b_j$, $j=1,2$  define the $\gamma_j$ by
$$\gamma_j(x_j)=1+i\sigma_j(x_j)$$ where the imaginary parts $\sigma_j$ are given by 
\[
  \sigma_j(x_j) = \left\{\begin{array}{@{}l@{\quad}l}
      0 & \mbox{if $|x_j|<a_j$,} \\[\jot]
      \sigma_0\left(\dfrac{|x_j|- a_j}{b_j-a_j}\right)^n& \mbox{if $a_j\leq |x_j|\leq b_j$,} \\[\jot]
   \sigma_0 & \mbox{if $| x_j|> b_j.$}
    \end{array}\right.
\]
where $n$ is a non-negative integer usually $n=0,1,2,3,4$ in applications. Other PML functions can be considered, e.g. unbounded 
functions \cite{BHPR04,BHPR08}. In this paper we will consider only  polynomial $\gamma_j$ defined above. We shall use the 
following notation  $A = \mbox{diag}\left(\dfrac{\gamma_2(x_2)}{\gamma_1(x_1)}, \dfrac{\gamma_1(x_1)}{\gamma_2(x_2)}\right)$ and 
$J = \gamma_1(x_1)\gamma_2(x_2)$.  We denote by $\Omega_{\!_F}=([-a_1,a_1]\times [-a_2,a_2])\backslash\overline {D}$ the region 
outside the scatterer, and by $\Omega_{\!_{PML}}$ the PML region. We refer to Figure~\ref{fig1} for an illustration of the 
geometric setting. 
\begin{figure}[!ht]
\begin{picture}(0,250)(-100,-30)
\thicklines
\unitlength=0.007in
\qbezier(200,220)(240,260)(260,200)
\qbezier(200,220)(160,260)(140,200)
\put(194,160){$D$}
\put(150,240){$\Gamma=\partial D$}
\put(200,140){\line(1,1){60}}
\put(200,140){\line(-1,1){60}}
\put(0,0){\line(1,0){400}}
\put(0,0){\line(0,1){400}}
\put(0,400){\line(1,0){400}}
\put(400,0){\line(0,1){400}}
\put(10,10){$\Omega_{\!_{PML}}$}
\put(350,10){$\Omega_{\!_{PML}}$}
\put(10,380){$\Omega_{\!_{PML}}$}
\put(350,370){$\Omega_{\!_{PML}}$}
\put(80,80){$\Omega_F$}
\put(50,50){\line(1,0){300}}
\put(50,50){\line(0,1){300}}
\put(350,50){\line(0,1){300}}
\put(50,350){\line(1,0){300}}
\put(100,410){$\Gamma_{\!_{PML}}$}
\put(200,60){$\Gamma_1= \partial \Omega_F$}
\end{picture}
\caption{An Illustration of the problem domain with a perfectly matched layer $\Omega_{\!_{PML}}$. The computational region of interest is the annulus $\Omega_{\!_F}$. 
 \label{fig1}} 
\end{figure}
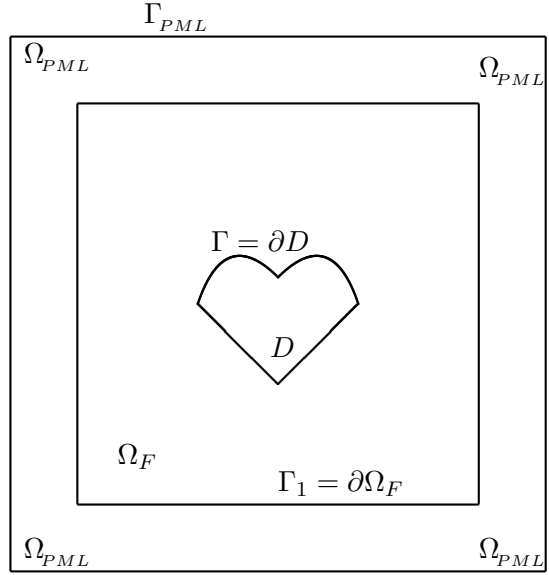
Then the exterior Helmholtz problem (\ref{exteriorP}) is approximated by the truncated PML problem in 
$\Omega:=\Omega_{\!_{PML}}\cup\Omega_{\!_F}$: Find $u_\PML\in H^1_{\Gamma}(\Omega):=\{v\in 
H^1(\Omega):\;v=g\;\mbox{on}\;\Gamma,\;v=0,\;\mbox{on}\;\Gamma_{\!_{PML}}\}$
\begin{eqnarray}
\label{eqn:trpml}
-\nabla\cdot A\nabla u_\PML - k^2 J u_\PML &=& f\;\; \hbox{ in } \Omega \nonumber \\
u_\PML &=& g\;\;\hbox{ on }\;\;\Gamma \nonumber \\
u_\PML &=& 0, \;\;\hbox{ on } \Gamma_{\!_{PML}}.
\end{eqnarray}

Note that inside $\Omega_{\!_F}$ we have $A=I$ and $J=1$ and the PML equation reduces to the standard Helmholtz equation. It is 
known that the truncated Cartesian PML problem above converges exponentially to the solution of the exterior Helmholtz equation 
inside the domain $\Omega_{\!_F}$ with a rate of convergence depending on the thickness of the PML layer and on the size of the 
parameter $\sigma_0$. Therefore it is sufficient to approximate the PML solution $u_\PML$ using bivariate splines. As the 
bivariate splines can efficiently and effectively solve the boundary value problem of Helmholtz equation as demonstrated
in \cite{LM18}, we apply the bivariate splines to approximate the PML solution.

\section{The Well-Posedness of the exterior domain Helmholtz equation}
In this section we first present the existence, uniqueness and stability of the truncated Cartesian PML 
problem. Then we explain the approximation of the PML solution $u_\PML$ to the exact solution $u$ of (\ref{exteriorP}). 

Let $\mathbb{L}^2(\Omega)$ be the space of all complex-valued square integrable functions over $\Omega$ and 
$\mathbb{H}^1(\Omega)\subset \mathbb{L}^2(\Omega)$ be the complex-valued square integrable functions over $\Omega$ 
such that their gradients are in $\mathbb{L}^2(\Omega)$. 
Let $\mathbb{H}^1_0(\Omega)$ be the space of complex valued functions 
with bounded gradients which vanish on the boundary of $\Omega$.    
The weak formulation of the PML exterior domain Helmholtz equation can be defined in a standard fashion as follows. Find 
$u_\PML\in\mathbb{H}^1_0(\Omega)$ that satisfies $u_\PML|_{\Gamma}=g$, $u_\PML|_{\Gamma_{\PML}}=0$ and 
\begin{eqnarray}
\label{weakform}
\int_{\Omega}\left(A\nabla u_\PML\cdot\nabla\overline{v} - k^2 J u_\PML\overline{v}\right)\; d\bfx &=& 
\int_{\Omega}f\overline{v}\;d\bfx
\end{eqnarray}
for all $v\in \mathbb{H}^1_0(\Omega)$. For convenience, let $B_{k^2}(u_\PML,v)$ be the sesquilinear form on the left hand side of 
(\ref{weakform}). The above equation can then be simply rewritten as 
\begin{equation}
\label{weakform2}
B_{k^2}(u_\PML, v)= \langle f, v\rangle_\Omega, \quad \forall v\in \mathbb{H}^1_0(\Omega). 
\end{equation}
where $\langle f, v\rangle_\Omega = \int_{\Omega} f \overline{v}d\bfx$. We shall use the Fredholm alternative 
theorem to establish the existence and uniqueness of the PML weak solution $u_\PML$. 

Consider the following two second order elliptic partial differential equations:
\begin{equation}
\label{eigen}
\begin{cases}
B_{\lambda}(u, v) &= 0, \hbox{ in } \Omega \cr
u &=0, \hbox{ in } \partial \Omega
\end{cases}
\end{equation}
and
\begin{equation}
\label{PDE2}
\begin{cases}
B_{\lambda}(u, v) &= \langle f, v\rangle_{\Omega}, \hbox{ in } \Omega \cr
u  &=0, \hbox{ in } \partial \Omega,
\end{cases}
\end{equation}
where $\lambda> 0$ is a constant, $f\in L^2(\Omega)$. We have the following well-known result: 
\begin{theorem}[Fredholm Alternative Theorem]
\label{Fredholm} Fix $\lambda>0$. Precisely one of the following two statements holds: 
Either (\ref{eigen}) has a nonzero weak solution $u\in \mathbb{H}^1(\Omega)$ 
or there exists a unique weak solution $u_f\in \mathbb{H}^1(\Omega)$ satisfying (\ref{PDE2}). 
\end{theorem}
\begin{proof}
We refer to \cite{E98} for a proof. 
\end{proof}

\begin{theorem} 
\label{mainresult}
Let $g\in  H^{-1/2}(\Gamma)$ and $f\in L^2(\Omega)$.  Suppose that 
$\sigma_j, j=1, 2$ be bounded, continuous, non-negative and 
monotonically increasing functions. Then, the truncated PML problem (\ref{eqn:trpml}) has a unique 
weak solution $u_\PML \in \mathbb{H}^1(\Omega)$ in the sense of (\ref{weakform2}) for every real value 
$k$, except at most for a discrete set of values of $k$.
\end{theorem}
\begin{proof}\;
If $k^2$ is not an eigenvalue of (\ref{eigen}), then we use 
Fredholm Alternative Theorem~\ref{Fredholm} to conclude that the PDE in (\ref{PDE2}) has a unique solution.  
Otherwise, it is well-known that there are only countable numbers of eigenvalues of the Dirichlet problem of Laplace operator 
and these eigenvalues $\lambda_i>0, i=1, 2, \cdots, \infty$ with $\lambda_j\to \infty$. 
\end{proof} 

Next we discuss the stability of the weak solution $u_\PML$. To analyze the weak formulation, it is convenient to introduce 
the following norm which is equivalent to the standard $H^1$ norm:
\begin{eqnarray}
\3bar u \3bar_{H}^2 
= &&\int_{\Omega}\left(\sum_{j=1}^2|A_{jj}|\left|\dfrac{\partial u}{\partial x_j}\right|^2+ k^2 |J||u|^2\right)\;d\bfx\\
=&&|u|^2_H+k^2\|u\|^2_H
\end{eqnarray}  
where  
$|u|^2_H = \int_{\Omega}\sum_{j=1}^2|A_{jj}|\left|\partial_{x_j} u\right|^2\;d\bfx$ the PML weighted $\mathbb{H}^1(\Omega)$ semi-norm and by $\|u\|^2_H = \int_{\Omega} |J||u|^2 d\bfx$ the weighted $\mathbb{L}^2(\Omega)$ norm.

The following continuity condition of the sesquilinear form $B_{k^2}(u, v)$ can be obtained easily. 
\begin{lemma}
\label{biform}
Suppose that $f\in L^2(\Omega)$. Then 
\begin{equation}
\label{continuity}
|B_{k^2}(u, v)| \le C_B \3bar  u\3bar_{H}  \, \3bar v\3bar_{H},
\end{equation} 
where $C_B$ is a positive constant dependent on $\Omega$ only. 
\end{lemma}

One of the major theoretical results in this paper is the following 
\begin{theorem}
\label{mainresult2}
Let $\Omega \subset \mathbb{R}^2$ be a bounded Lipschitz domain. Suppose that $k^2$ is not a Dirichlet 
eigenvalue of (\ref{eigen}) over $\Omega$. Then there exists $C_*>0$ which does not go to zero when $k\to \infty$ 
such that 
\begin{equation}
\label{infsup2}
\inf_{v\in \mathbb{H}^1(\Omega)}  \sup_{u\in \mathbb{H}^1(\Omega)}  
\frac{\hbox{Re}(B_{k^2}(u,v))}{\3bar u\3bar_{H}\, \3bar v\3bar_{H}} \ge C_*.  
\end{equation}
\end{theorem}  
\begin{proof}
Suppose (\ref{infsup2}) does not hold. Then there exists $v_n\in \mathbb{H}^1(\Omega)$ such that $\3bar v_n\3bar_{H}=1$ and 
$$
\sup_{u\in \mathbb{H}^1(\Omega)}  
\frac{\hbox{Re}(B_{k^2}(u,v_n))}{\3bar u\3bar_{H}} \le \frac{1}{n}
$$
for $n=1, \cdots, \infty$. The boundedness of $v_n$ implies that there exists a convergent subsequence that converges in $\mathbb{L}^2(\Omega)$, by the
Rellich-Kondrachov Theorem. Without loss of generality we may assume that $v_n\to v^*$ 
in the $\mathbb{L}^2(\Omega)$ norm and in the semi-norm on $\mathbb{H}^1(\Omega)$ with 
$\3bar v^*\3bar_{H}=1$. It follows that for each 
$u\in  \mathbb{H}^1(\Omega)$ with $\3bar u\3bar_{H}=1$, 
$\hbox{Re}(B_{k^2}(u,v_n)) \to 0$. Hence, $\hbox{Re}(B_{k^2}(u,v^*))=0$ 
for all $u\in \mathbb{H}^1(\Omega)$.  
By using $u= v^*$, we see that $B_{k^2}(v^*, v^*))=0$. 

So if $v^*\not=0$,  $v^*\in \mathbb{H}^1_0(\Omega)$ is an eigenfunction with eigenvalue $k^2$ of the Dirichlet problem for
the truncated homogeneous PML problem (\ref{eigen}). This contradicts the assumption that $k^2$ is not a Dirichlet eigenvalue of (\ref{eigen}).
Hence, we have $v^*\equiv 0$ which contradicts to the fact $\3bar v^*\3bar_{H}=1$.  Therefore, $C_*>0$ in (\ref{infsup2}).  

We next show that $C_*$ does not go to zero  when $k\to \infty$.  
For each integer $k>0$, let $L_k$ be the best lower bound in (\ref{infsup2}). 
We claim that $L_k\not \to 0$. Indeed, for $L_k>0$, we can find $v_k$ with $\3bar v_k\3bar_{H}=1$ such that 
$$
\sup_{u\in \mathbb{H}^1(\Omega)}  
\frac{ \hbox{Re}(B_{k^2}(u,v_k))}{\3bar u\3bar_{H}} \le 2L_k. 
$$  
That is, $\hbox{Re}(B_{k^2}(v_k,v_k)) \le 2L_k$.  
Since  $|v_k|_H\le 1$ and $k^2 \|v_k\|_H^2 \le 1$ or $\|v_k\|_H\le 1/k\le 1$, we use Rellich-Kondrachov 
Theorem again to conclude that there exists a $u^*\in \mathbb{H}^1(\Omega)$ such that $v_k\to u^*$ in $L^2$ norm 
and $|v_k|_H\to |u^*|_H$ without loss of generality.  
As $\|v_k\|_H\le 1/k$, we have $\|u^*\|_H\le 2/k$ for $k>0$ large enough. It follows that $u^*\equiv 0$. 
Thus, $\nabla u^*\equiv 0$, that is, $|u^*|_H=0$. Hence, $|v_k|_H\to 0$. If $L_k\to 0$, we have
$ \large{|}| v_k|_H^2 -k^2 \|v_k\|_H^2\large{|} = |\hbox{Re}(B_{k^2}(v_k, v_k))| \to 0$. 
It follows that  $k^2\|v_k\|_H^2 \to 0$. However, 
since $\3bar v_k\3bar_{H}=1$, we have $k^2\|v_k\|_H^2 \to 1$.  That is, we got a contradiction.   
Therefore, $L_k$ does not go to zero when $k\to \infty$.    
\end{proof}

Furthermore, the weak solution is stable. Indeed, 
\begin{theorem}
\label{newstable}
Let $\Omega \subset \mathbb{R}^2$ be a bounded Lipschitz domain as described in the introduction section. 
Suppose that $k^2$ is not a Dirichlet eigenvalue of the elliptic PDE (\ref{eigen}) over $\Omega$. 
Let $u_\PML\in \mathbb{H}^1(\Omega)$ be the unique weak solution to (\ref{weakform}) as in Theorem~\ref{mainresult}. 
Then there exists a constant $C>0$ independent of $f$ such that 
\begin{equation}
\label{newbound}
\3bar u_\PML \3bar_{H} \le C\|f\|
\end{equation}
for $k\ge 1$, where $C$ is dependent on  the constant $C_*$ which is the lower bound in (\ref{infsup2}) and hence, 
$C$ will not go to $\infty$ as $k\to \infty$.  
\end{theorem}
\begin{proof} Let $u_\PML$ be the weak solution of the exterior domain Helmholtz equation satisfying (\ref{exteriorP}).  
By using (\ref{infsup2}), we have
\begin{align}
C_*\3bar u_\PML \3bar_{H} &\le \sup_{u\in \mathbb{H}^1(\Omega)}  
\frac{ \hbox{Re}(B_{k^2}(u,u_\PML))}{\3bar u\3bar_{H}}  
= \sup_{u\in \mathbb{H}^1(\Omega)}  
\frac{ \hbox{Re}(\langle u, f\rangle_\Omega}{\3bar u\3bar_{H}} \le \|f\|  
\end{align}
by using Cauchy-Schwarz inequality in the last step. 
\end{proof}

We refer to \cite{BP13, KP10a} for more properties of the weak solution $u_\PML$. 
In particular it is known that  $u_\PML$ approximates the solution $u$ of exterior domain Helmholtz equation 
(\ref{exteriorP}) very well. 
Indeed, in \cite{BP13}, we saw if $\sigma_0$ is the PML strength, and $M=b-a$ is the size of the PML layer with stretching 
functions in the $x_1$ and $x_2$ directions equal, for convenience, then the truncated PML problem is stable provided that 
$M\sigma_0$ is sufficiently large. Moreover, the solution of the truncated PML problem \ref{eqn:pmlf} converges exponentially to 
the solution of the exterior Helmholtz problem \ref{exteriorP}. More precisely, the following result can be found in 
\cite{BP13}. 

\begin{theorem}[Theorem 5.8 of \cite{BP13}]
\label{BPThm} 
Suppose that $\sigma_0M$ is large enough. Let $g\in H^{1/2}(\Gamma)$ be given, $u$ be the solution of the exterior 
Helmholtz problem \eqref{exteriorP} with $f\equiv 0$ and $u_{\PML}$ be the solution of the truncated PML problem. 
Then there exist constants $c>0, C>0$ such that,
\begin{equation}
\label{BPest}
\|u- u_\PML\|_{H^1(\Omega)} \leq C e^{-c\sigma_0M}\|g\|_{H^{1/2}(\Gamma)}.
\end{equation}
\end{theorem}

The result of Theorem \ref{BPThm} justifies the use of the PML technique to truncate the exterior Helmholtz problem. Furthermore, 
Theorem \ref{BPThm} also shows that good accuracy can be achieved by simply increasing the value of the PML parameter $\sigma_0$ 
on a fixed grid. Numerical behavior of $\sigma$ vs. $M$ can be found in a later section.


\section{Bivariate Spline Spaces and Spline Solutions}
Let us begin by reviewing some basic properties of bivariate spline spaces  
which will be useful in the study in this paper. We refer to  \cite{LS07} and \cite{ALW06} for
details. Given a polygonal region $\Omega:=\Omega_\PML$, a collection $\triangle:=\{T_{1},...,T_{n}\}$ of triangles is an ordinary 
triangulation of $\Omega$ if $\Omega=\cup_{i=1}^{n}T_{i}$ and if any two triangles $T_{i}, T_{j}$ intersect at 
most at a common vertex or a common edge. We say $\triangle$ is $\beta$ quasi-uniform if
\begin{equation}
\max_{T\in \triangle} \frac{|T|}{\rho_T}\le \beta<\infty
\end{equation}
where $\rho_T$ is the radius of the inscribed circle of $T\in \triangle$.
For $r\ge 0$ and $d>r$, let 
\begin{equation}
\label{splinespace}
S^r_d(\triangle)=\{s\in C^r(\Omega):  s|_T\in \mathbb{P}_d, \forall T\in \triangle\}
\end{equation} 
be the spline space of degree $d$ and smoothness $r\ge 0$ over triangulation $\triangle$.  

Let us quickly explain the structure of spline functions in $S^r_d(\triangle)$. 
For a triangle $T_i\in \triangle$, $T_{i}=(v_{1},v_{2},v_3)$, we define the barycentric coordinates $(b_1,b_2,b_3)$
of a point $(x_o,y_o)\in \Omega$. These coordinates are the solution of the following system of equations
\begin{align*}
&b_{1}+b_{2}+b_{3}=1\\
&b_{1}v_{1,x}+b_{2}v_{2,x}+b_{3}v_{3,x}=x_{o}\\
&b_{1}v_{1,y}+b_{2}v_{2,y}+b_{3}v_{3,y}=y_{o},
\end{align*}
and are nonnegative if $(x_o,y_o)\in T_i$, where $v_{1,x}$ denotes the $x$ coordinate of vertex $v_1$. 
The barycentric coordinates are then used to define the Bernstein polynomials of degree $d$:
\begin{equation*}
B^{T}_{i,j,k}(x,y):=\frac{d!}{i!j!k!}b_1^ib_2^jb_3^k, \hspace*{1cm} i+j+k=d.
\end{equation*}
which form a basis for the space $\mathcal{P}_d$ of polynomials of degree $d$.  
Therefore we can represent all $P\in \mathcal{P}_d$ in B-form:
\begin{equation*}
P=\sum_{i+j+k=d}p_{ijk}B^T_{ijk},
\end{equation*}
where the $B$-coefficients $p_{ijk}$ are uniquely determined by $P$. 

We define the spline space $S^{-1}_d:=\lbrace s|_{T_i} \in \mathcal{P}_d \rbrace, T_i\in \triangle$,
 where $T_i$ is a triangle in a triangulation $\triangle$ of $\Omega$. We use this piecewise continuous 
polynomial space to define 
\begin{equation}
\label{splinespace1}
 S^{r}_{d}:=C^{r}(\Omega) \cap S^{0}_{d}(\triangle),
 \end{equation} 

See properties and implementation of these spline spaces in \cite{LS07} and a method of implementation of multivariate splines
in \cite{ALW06}.  
We include information about spline smoothness here which is significantly different from the finite element method and
discontinuous Galerkin method. Note that the implementation of bivariate splines in \cite{S15} is similar to 
the finite element.

We start with the recurrence relation (de Castlejau algorithm)
\begin{align*}
B^d_{ijk}=b_1B^{d-l}_{i-1,j,k}+b_2B^{d-l}_{i,j-1,k}+b_3B^{d-l}_{i,j,k-1},\quad\textnormal{for all }i+j+k=d,
\end{align*}
where all items with negative subscripts are taken to be 0, allows us to define $\bfc^{(0)} =\bfc$ with 
\begin{align*}
\bfc^{(0)}_{ijk}:=p_{ijk}, i+j+k=d.
\end{align*}
Then for $\l=1,...,d$, we have $\bfc^{\l}$ with 
\begin{align*}
\bfc^{(\l)}_{ijk}=b_1 \bfc^{(\l-1)}_{i+1,j,k}+b_2 \bfc^{(\l-1)}_{i,j+1,k}+b_3 \bfc^{(\l-1)}_{i,j,k+1},
\end{align*}
so, letting $u=(x,y)$, we can write
\begin{align*}
P(u)=\sum_{i+j+k=d-\l}\bfc^{(\l)}_{ijk}B^{d-\l}_{ijk}(u).
\end{align*}
Let $u, v\in\mathbb{R}^2$ be represented in barycentric coordinates by $(\alpha_1,\alpha_2,\alpha_3)$ and 
$(\beta_1,\beta_2,\beta_3)$ respectively. Then the vector $a=u-v$ is given in barycentric coordinates by $a_i=\alpha_i-\beta_i$, 
and the derivative in that direction is given by 
\begin{align}\label{splder}
a\cdot \nabla B_{ijk}^{d}(v)=d\big(a_1B^{d-l}_{i-1,j,k}+a_2B^{d-l}_{i,j-1,k}+a_3B^{d-l}_{i,j,k-1}\big)
\end{align}
for any $i+j+k=d$. With the above, we can derive smoothness conditions as follows.  

Consider two triangles $T_1=(v_1,v_2,v_3)$ and $T_2=(v_4,v_3,v_2)$ joined along the edge $e_{23}=\langle v_2, v_3\rangle$. Let
\begin{align*}
p(x,y):=\sum_{i+j+k=d}c_{ijk}B^d_{ijk}(x,y),\quad\quad q(x,y):=\sum_{i+j+k=d}r_{ijk}R^d_{ijk}(x,y)
\end{align*}
To enforce continuity of the spline consisting of $p$ and $q$ over domain $T_1\cup T_2$, we enforce 
$p(x,y)|_{e_{23}}=q(x,y)|_{e_{23}}$ which leads to the condition $c_0jk=r_0kj$ for all $j+k=d$. 
This also guarantees that the derivatives in a direction tangent to $e_{23}$ will match.  Next 
we take a direction $\bfa$ not parallel to $e_{23}$ whose barycentric coordinates with respect to $T_1$ and $T_2$ 
are denoted $\alpha$ and $\beta$, respectively. We require 
\begin{align*}
\bfa\cdot\nabla p(x,y)|_{e_{23}}&=\bfa\cdot\nabla q(x,y)|_{e_{23}}\\ 
\hbox{ or } & \\
d \sum_{j+k=d-1}{\bf c}^{(1)}_{0jk}(\alpha)B_{0jk}^{d-1}(v)&=d \sum_{j+k=d-1}{\bf r}^{(1)}_{0jk}(\beta)R_{0jk}^{d-1}(v),
\end{align*}
where ${\bf c}$, ${\bf r}$ are iterates of the de Casteljau algorithm mentioned above. 

But on $e_{23}$ we have $B_{0jk}=R_{0kj}$, so it follows that the derivatives of $p$ and $q$ will match on $e_{23}$ if and only if
\begin{align*}
{\bf c}^{(1)}_{0jk}(\alpha)={\bf r}^{(1)}_{0kj}(\beta).
\end{align*}
 See more detail in \cite{LS07} for higher order smoothness. 
 
It is clear to see that the smoothness conditions are linear 
equations in terms of coefficient vector $\bfc=[c_{ijk}, i+j+k=d]\cup [r_{ijk}, i+j+k=d]$. 
The $C^0$ and $C^1$ smoothness conditions can be equivalently written as linear system $H \bfc=0$. 
In general, we collect the coefficients of all polynomials over triangles in $\triangle$ and put them 
together to form a vector $\bfc$ and put all smoothness conditions together to form a set of smoothness 
constraints $H\bfc = 0$.  Certainly, more detail can be found in \cite{LS07}.  We refer \cite{ALW06} 
for how to implement these
spline functions for numerical solution of some basic partial differential equations. See 
\cite{LW04}, \cite{HHL07}, \cite{GLS15}, \cite{LW18}, \cite{LM18} for spline solution to other PDEs.

As solutions to the Helmholtz equation will be 
a complex valued solution, let us use a complex spline space in this paper defined by 
\begin{equation}
\label{splinespace2}
\mathbb{S}^r_p(\triangle)=\{s= s_r +\bfi s_i, s_i, s_r\in S^r_p(\triangle)\}.
\end{equation} 

The complex spline space $\mathbb{S}^r_p(\triangle)$ has the similar approximation properties as the standard  real-valued spline
space $S^r_p(\triangle)$. The following theorem can be established by the same  
constructional techniques (cf.  \cite{LS07} for spline space $S^r_p(\triangle)$ for real valued functions): 
\begin{theorem}
\label{SplineOrder}
Suppose that $\triangle$ is a $\gamma$-quasi-uniform triangulation of polygonal domain $\Omega$. 
Let $p\ge 3r+2$ be the degree of spline space $\mathbb{S}_p^r(\triangle)$.  
For every $u\in  \mathbb{H}^{m+1}(\Omega)$, there exists a  quasi-interpolatory spline function 
$Q_p(u)\in \mathbb{S}^r_p(\triangle)$ such that 
\begin{equation}
\label{LSapp2}
\sum_{T\in \triangle} \|D^\alpha_x D^\beta_y (u- Q_p(u))\|^2_{2,T}
\le K_5 |\triangle|^{2(m+1-s)} |u|^2_{2,p+1,\Omega}
\end{equation}
for $\alpha+\beta=s, 0\le s\le m+1$, where $0\le m\le p$,  $K_5$ is a positive constant dependent only on $\gamma$, $\Omega$, 
and $p$.
\end{theorem}

With the above preparation, we now introduce  spline weak solution to (\ref{exteriorP}).   Let   
$s_\PML \in \mathbb{S}^r_p(\triangle)$ be a spline function satisfying $s_{\triangle}|_{\Gamma}=g$ and 
the weak formulation: 
\begin{equation}
\label{splineweak}
B_{k^2}(s_\triangle, v)=\langle f, v\rangle_\Omega, \quad \forall v\in \mathbb{S}^r_p(\triangle)\cap \mathbb{H}^1_0(\Omega).
\end{equation}
$s_\PML$ is called a PML solution of (\ref{exteriorP}).  We will show that 
$s_\PML$ approximates the PML solution $u_\PML$ discussed in the previous section.     

To do so, we first explain the well-posedness of $s_\PML$.  
Similar to Theorem~\ref{mainresult} and \ref{mainresult2} in the previous section, we have
\begin{theorem}
\label{mainresult3}
Let $\Omega \subset \mathbb{R}^2$ be a bounded Lipschitz domain. Suppose that $k^2$ is not a Dirichlet 
eigenvalue of (\ref{eigen}) over $\Omega$. Let $\mathbb{S}^1_p(\triangle)$ be the complex valued spline space 
defined above, where $\triangle$ is a triangulation of $\Omega$.  
Then there exists a unique spline weak solution $s_\PML\in \mathbb{S}^1_p(\triangle)$ satisfying 
(\ref{splineweak}). Furthermore, $s_\PML$ is stable in the sense that 
$$
\3bar s_\PML \3bar_{H}\le C \|f\|_{2,\Omega}.
$$
where $C>0$ is a constant does not go to $\infty$ when $k\to \infty$.  
\end{theorem}  

Next we discuss the approximation of $s_\PML$ to the PML solution $u_\PML$.  
The orthogonality condition follows easily: 
\begin{equation}
\label{orthogonality}
B_{k^2}(u_\PML -s_\PML, v) =0, \quad \forall v\in \mathbb{S}^r_p(\triangle) \cap \mathbb{H}^1_0(\Omega). 
\end{equation}

For convenience, let us first consider homogeneous Helmholtz equation. That is, $f\equiv 0$ in (\ref{exteriorP}).
Thus we are ready to prove one of  the following main results in this paper. 
\begin{theorem}
\label{mainresult8}
Let $\Omega$ be a bounded Lipschitz domain as defined in the introduction.  
Suppose that $k^2$ is not a Dirichlet eigenvalue of the elliptic PDE in (\ref{eigen}). 
Let $u_\PML$ be the unique weak solution in $\mathbb{H}^1(\Omega)$ satisfying (\ref{weakform}) and $s_\PML\in  
\mathbb{S}^r_p(\triangle)$ be the spline weak solution satisfying  (\ref{splineweakform}). 
Suppose that the exact solution $u\in \mathbb{H}^{s}(\Omega)$ of (\ref{exteriorP}) with 
$1\le s\le p$. Then there exists a positive constant $C>0$ which does not go to $\infty$ when $k\to \infty$ such that 
\begin{equation}
\label{convergencerate4}
\3bar u - s_\PML \3bar_{H} \le C\exp(-c\sigma_0 M)\|g\|_{H^{1/2}(\Gamma)} +  
C (1+ k|\triangle|)|\triangle|^{s-1}  |u_\PML|_{s,2,\Omega}, 
\end{equation}
where $|u_\PML|_{s,2, \Omega}$ is the semi-norm in $\mathbb{H}^{s}(\Omega)$. 
\end{theorem}
\begin{proof}
We first apply the lower bound in Theorem~\ref{mainresult} to $u - s_\PML$ and have
\begin{equation}
\label{keystep0}
C_* \3bar u_\PML - s_\PML \3bar^2_{H} \le  |B_{k^2}(u_\PML - s_\PML, u_\PML - s_\PML)|.  
\end{equation}
Next we use the orthogonality condition (\ref{orthogonality}) to have 
\begin{eqnarray*}
|B_{k^2}(u_\PML - s_\PML, u_\PML - s_\PML)| 
&=& |B_{k^2}(u_\PML - s_\PML, u_\PML- Q_p(u_\PML)| \cr
&\le &  C_B \3bar u_\PML - s_\PML \3bar_{H} \3bar u_\PML- Q_p(u_\PML)\3bar_{H} 
\end{eqnarray*}
by using Lemma~\ref{biform}, i.e. (\ref{continuity}), 
where  $Q_p(u_\PML)$ is the quasi-interpolatory spline of $u_\PML$ as  in Theorem~\ref{SplineOrder}.
It follows  that 
\begin{equation}
\label{finalterm}
\3bar  u_\PML - s_\PML \3bar_{H} \le C_*^{-1}C_B  \3bar u_\PML- Q_p(u_\PML)\3bar_{H}.
\end{equation}
Finally, we use the approximation property of spline space $\mathbb{S}^r_p(\triangle)$, i.e. (\ref{LSapp2}). 
For $u_\PML\in \mathbb{H}^{s}(\Omega)$ with $1\le s\le p$, 
we use the quasi-interpolatory operator $Q_p(u_\PML)$  to have 
$$
\3bar u_\PML- Q_p(u_\PML)\3bar_{H}\le C (1+ k|\triangle|) |\triangle|^{s-1} |u_\PML|_{s,2,\Omega} 
$$
for a constant $C$ dependent on $\omega$, $p$ and the smallest angle of $\triangle$ only.  

Now we can use  Theorem~\ref{BPThm} to finish the proof. 
Indeed, we combine the estimates above  with the estimate in \ref{BPest} to have
\begin{eqnarray}
\label{convergencerate8}
\3bar u - s_\PML \3bar_{H} &\le& \3bar u - u_\PML \3bar_H + \3bar u_\PML - s_\PML\3bar_H \cr
&\le & C_1\exp(-cM \sigma_0) \|g\|_{H^{1/2}(\Gamma)} + 
C_2(1+ k|\triangle|)|\triangle|^{s-1}  |u_\PML|_{s,2,\Omega}, 
\end{eqnarray}
where $c>0, C_1$ are positive constants as in Theorem~\ref{BPThm} and $C_2>0$ is another positive constant 
dependent on the lower bound $C_*$ and the approximation property in Theorem~\ref{SplineOrder}.  
\end{proof} 

Next let $u_{f,g}$ be the exact solution of (\ref{exteriorP}) and let $u_{f,0}$ be the weak solution of 
\begin{eqnarray}
& -\Delta u-k^2 u = f\;\; &\mbox{in}\;\;\mathbb{R}^2\setminus\overline{D} \nonumber \\
& u=0\;\; &\mbox{on}\;\;\Gamma:=\partial D, \label{exteriorPH} \\ 
&\lim_{r\rightarrow\infty}\sqrt{r}\left(\dfrac{\partial u}{\partial r} -iku\right)=0,\;\;& r=|\bfx|,\;\bfx\in\mathbb{R}^2. \nonumber 
\end{eqnarray}
by using Theorem~\ref{Fredholm}. We extend $u_{f,0}$ outside of $\Omega$ by zero naturally. We know  
$u=u_{f,g}- u_{f,0}$ satisfies the exterior domain problem (\ref{exteriorP}) 
of Helmholtz equation with Sommerfeld radiation condition at $\infty$. 
Then by Theorem~\ref{BPThm}, there exists a PML solution $u_{PML}$ which converges to $u$ very well satisfying  
\eqref{BPest}. Let $s_{f,0}$ be the spline approximation of the elliptic PDE (\ref{exteriorPH}). In fact, 
$s_{f,0}$ is the spline solution to the following boundary value problem:
\begin{eqnarray}
& -\Delta u-k^2 u = f\;\; &\mbox{in}\;\;\Omega, \nonumber \\
& u=0\;\; &\mbox{on}\;\;\partial \Omega. \label{exteriorPH2}
\end{eqnarray}
According to \cite{ALW06}, we can find a spline approximation $s_{f,0}$ to the exact solution $u$ of
the PDE above as long as $k^2$ is not an eigenvalue of Dirichlet problem of the Laplace operator over $\Omega$. 
By extending $s_{f,0}$ outside of $\Omega$ by zero, we see that $s_{f,0}$ is a good approximation of $u_{f,0}$
satisfying the Sommerfeld radiation condition in  (\ref{exteriorPH}).   

Let $u_\PML$ be the PML 
approximation of (\ref{exteriorP}) and $s_\PML$ be the spline approximation of (\ref{exteriorP}) with $f=0$. 
Let 
\begin{equation}
\label{splinesolution}
s_\triangle =  s_{f,0}+ s_\PML
\end{equation}
be the spline weak solution to the original PDE (\ref{exteriorP}).  Then we can show 
\begin{eqnarray*}
 \|u_{f,g}- s_\triangle\| &=&\| u_{f,g} -  s_{f,0}- s_\PML \|= \|u_{f,g}- u_{f,0} - u_\PML + \|u_{f,0}- s_{f,0} +
u_\PML- s_\PML \|\cr 
   &\le & \| u - u_\PML \| + \|u_{f,0} - s_{f,0}\| + \|u_\PML -  s_\PML \|.
\end{eqnarray*}
We use Theorem~\ref{BPThm} to the first term, Theorem 8 in \cite{ALW06} to the middle term, and
a part of the proof of Theorem~\ref{mainresult8} to the last term on the right-hand side of the above estimate.  

\begin{theorem}
\label{mainresult9}
Let $\Omega$ be a bounded Lipschitz domain as defined in the introduction.  
Suppose that $k^2$ is not a Dirichlet eigenvalue of the elliptic PDE in (\ref{eigen}), nor a Dirichlet eigenvalue
of the Laplace operator over $\Omega$.  
Let $u_\PML$ be the unique weak solution in $\mathbb{H}^1(\Omega)$ satisfying (\ref{weakform2}) and $s_\PML\in  
\mathbb{S}^r_p(\triangle)$ be the spline weak solution satisfying  (\ref{splineweakform}). 
Suppose that the exact solution $u_{f,g}\in \mathbb{H}^{s}(\Omega)$ of (\ref{exteriorP}) with 
$1\le s\le p$. Similarly, suppose that the exact solution $u_{f,0}\in \mathbb{H}^{s}(\Omega)$ of 
(\ref{exteriorPH}). 
Then there exists a positive constant $C>0$ which does not go to $\infty$ when $k\to \infty$ such that 
\begin{equation}
\label{convergencerate5}
\3bar u_{f,g} - s_\triangle \3bar_{H} \le C\exp(-c\sigma_0 M)\|g\|_{H^{1/2}(\Gamma)} + 
 C (1+ k|\triangle|)|\triangle|^{s-1}  |u_{f,0}|_{s,2,\Omega} +
C (1+ k|\triangle|)|\triangle|^{s-1}  |u_\PML|_{s,2,\Omega}, 
\end{equation}
where $|u_\PML|_{s,2, \Omega}$ is the semi-norm in $\mathbb{H}^{s}(\Omega)$. 
\end{theorem}

\section{Numerical Implementation}
The implementation of bivariate splines is different from the traditional finite element method and
the spline method in \cite{S15}. It builds a 
system of equations with smoothness constraints (across interior edges) and boundary condition constraints 
and then solves a constrained minimization instead of coding all smoothness  and boundary condition constraints
into basis functions and then solving a linear system of smaller size.  The constrained minimization is solved
using an iterative method with a very few iterations.   See an explanation of the implementation  
in \cite{ALW06}.  Several advantages of the implementation in \cite{ALW06} are: 
\begin{itemize}
\item{(1)} piecewise polynomials of any degree $p\ge 1$  and higher order smoothness can be used easily. Certainly, they
are subject to the memory of a computer. We usually use $p=5$ and often use $p=5, \cdots, 15$ together with the uniform 
refinement or adaptive refinement to obtain the best possible solution within the budget of computer memory.     
\item{(2)} there is no quadrature formula. The computation of the related inner product and triple product is done using formula; 
Note that the right-hand side function $f$ is approximated by using a continuous spline if $f$ is continuous or a discontinuous
spline function $s_f\in \mathbb{S}^{-1}_p(\triangle)$. 
\item{(3)} the computation of the mass matrix and weighted stiffness matrix is done in parallel; 
\item{(4)} the matrices for the smoothness conditions (across interior edges) and boundary 
conditions are done in parallel and are used as side constraints during the minimization of the discrete 
partial differential equations. 
\item{(5)} Using spline functions of higher degree than 1, 
the bivariate spline method can find accurate solution without using 
a  pre-conditioner as shown in Example 4 in the next section.
\end{itemize}     

In addition to triangulations, the implementation in \cite{ALW06} has been extended to the setting of  
polygonal splines of arbitrary degrees over a partition of polygons of any sides (cf. \cite{FL16} and 
\cite{LL17}).

Our computational algorithm is given as follows. For spline space $\mathbb{S}^1_p(\triangle)$, let
$\bfc$ be the coefficient vector associated with a spline function $s\in \mathbb{S}^{-1}_p(\triangle)$ 
which satisfies the smoothness conditions $H\bfc=0$ so that $s\in \mathbb{S}^1_p(\triangle)$. More precisely,  
as in the implementation explained in \cite{ALW06}, $\bfc$ is a stack of the polynomial coefficients over each
triangle in $\triangle$. That is, $\bfc$ is the vector for a spline function $s\in \mathbb{S}^{-1}_p(\triangle)$
a discontinuous piecewise polynomial function over $\triangle$.   
Let $H$ be the smoothness matrix such that $s\in \mathbb{S}^1_p(\triangle)$ if and
only if $H\bfc=0$.  

In our computation, we approximate $A_{ii}, i=1, 2$, $J$, and $f$ by continuous spline functions in 
$\mathbb{S}^0_p(\triangle)$ as $A_{ii}, J$ are continuously differentiable. For example, we use a good 
interpolation method to 
find a polynomial interpolation of $J$ over each triangle $T$ and these interpolatory polynomials over all triangles in 
$\triangle$ form a continuous spline in $\mathbb{S}^0_p(\triangle)$. We denote the interpolatory spline of $J$ by $\tilde{J}$.
Similar for each entry in $A$. It is easy to see that $\tilde{A}$ and $\tilde{J}$ approximate $A$ and $J$ 
very well. See \cite{LS07}.
In our computation, we really use the spline weak solution $s_\triangle$ which satisfies the following 
weak formulation:
\begin{eqnarray}
\label{splineweakform}
&&\int_{\Omega}\left[\tilde{A}\nabla s_{\!_\Delta}\cdot\nabla v - k^2\tilde{J} s_{\!_\Delta} v\right]d\bfx = \int_{\Omega} 
s_{\!_F}v d\bfx
\end{eqnarray}
for all testing spline functions $v\in {\cal S}^{-1}_d(\triangle)$ which vanish on $\Gamma_{\!_{PML}}$, and satisfy the 
boundary condition $v = s_g$ on the boundary of the scatterer $\Gamma$, with $\tilde{A}$, $\tilde{J}$ the spline approximations 
of $A$ and $J$ respectively. Note that this approximate weak form (\ref{splineweakform}) enables us to find the exact 
integrations in the (\ref{splineweakform}) without using any quadrature formulas.   

Next let $D$ be the Dirichlet boundary condition matrix such that the linear system $D\bfc  =\bfg$ is a discretization of 
the Dirichlet boundary condition, e.g. using interpolation at $p+1$ equally-spaced points over the boundary edge(s) of 
each boundary triangle, where $\bfg$ is a vector consisting of boundary values $g$ on $\Gamma$. Let $M$ and $K$ be the mass 
and stiffness matrices associated with the integrations in (\ref{splineweakform}). 
Then the spline solution to (\ref{splineweakform}) can be given in terms of these matrices as follows: 
\begin{equation}
\label{splineweakform2}
\bfc_\triangle^\top K \overline{\bfc} - k^2\bfc_\triangle^\top M\overline{\bfc} 
=\bff^\top M \overline{\bfc}  +\bfg^\top M_\Gamma \bfc, 
\quad \forall \bfc \in \mathbb{C}^N
\end{equation}
for $\bfc, \bfc_\triangle$ which satisfies $H\bfc_\triangle=0$ and $D\bfc_\triangle= \bfg$ 
while $H\bfc_\triangle=0$, 
where $N$ is the dimension of spline space $\mathbb{S}^1_p(\triangle)$. We now further weaken the formulation in 
(\ref{splineweakform2}) to require (\ref{splineweakform2}) holds for all splines in $\mathbb{S}^{-1}_p(\triangle)$.
 
To solve this constrained systems of 
linear equations, we use the so-called  the constrained iterative minimization method in \cite{ALW06}. We only
do a few iterative steps.  It makes our spline solutions very accurate for the Helmholtz problem with  
high wave numbers.  See numerical results in the next section and the numerical results in \cite{LM18} 
for boundary value problems of the Helmholtz equation.

In numerical experiments, we are faced with the issue of choosing the PML strength $\sigma_0$ and PML width $M$. Ideally, $M$ should be as small as possible for a given accuracy level in order to control the computational cost of the PML layer. The result of Theorem~\ref{BPThm} for the continuous PML problem show that the exponential rate of convergence is proportional to $\sigma_0 M$. In Fig~\ref{tr2}, we
investigate how the accuracy of the spline solution for the PML problem depends on the PML strength $\sigma_0$ and the PML width $M$. The domain is an annulus with a square hole of size $0.5\times 0.5$. The interior PML boundary is a square of size $2\times 2$ and the outer boundary is a square of size $(2+M)\times (2+M)$, where $M$ varies from $0.1$ to $1$, and the PML strength $\sigma$ varies from $0$ to $35$. The exact solution is $u(r,\theta) = e^{i\theta}H^{(1)}_1(kr)$ for $k=2\pi$ and for spline degrees $d=3,5,8$. The grid size is fixed at $h=0.1$ in all examples.
\begin{figure}[t!]\small
\begin{tabular}{cc}
\includegraphics[height=2in]{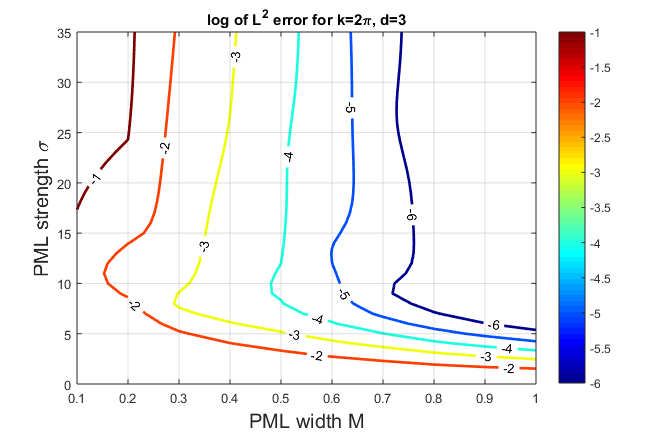} & \includegraphics[height=2in]{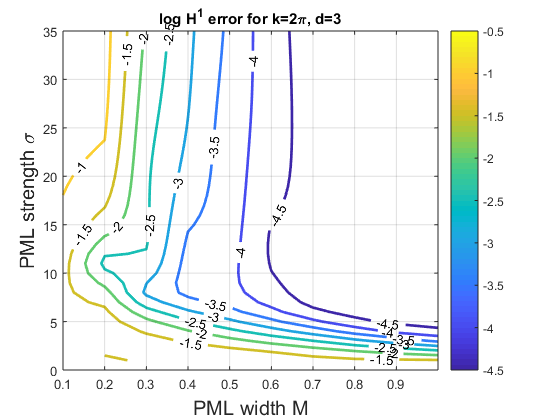} \cr
\includegraphics[height=2in]{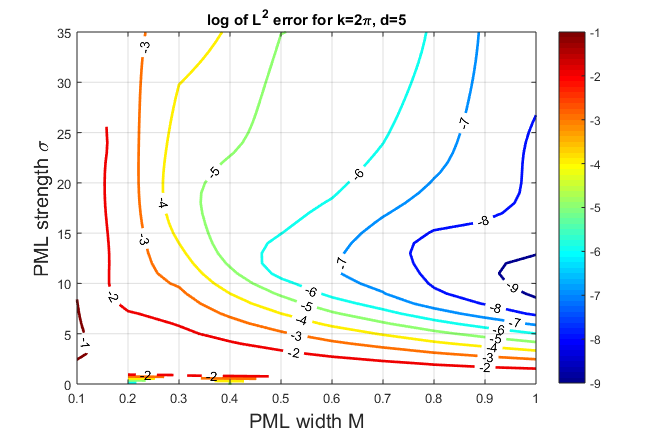} & \includegraphics[height=2in]{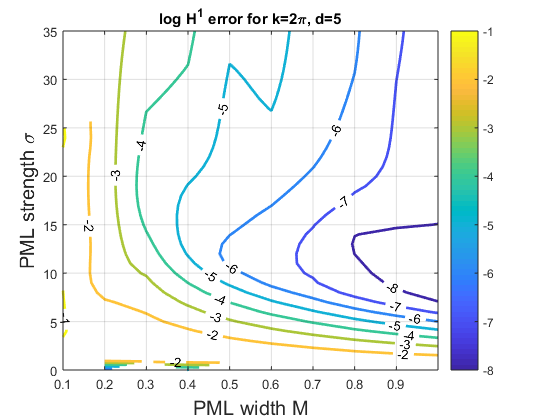} \cr 
\includegraphics[height=2in]{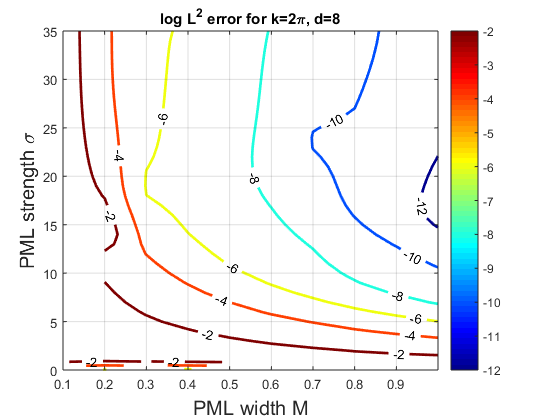} & \includegraphics[height=2in]{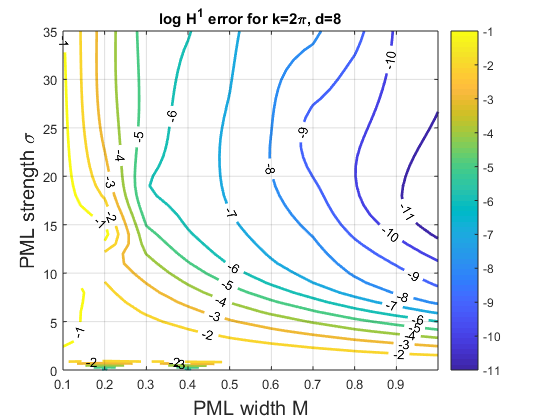} \cr
\end{tabular}
\caption{Contour plots for the $L^2$ and $H^1$ errors on a log scale for different PML strengths $\sigma_0$ and PML widths using splines of degrees $d=3,5,8$.
\label{tr2}}
\end{figure}
The results shown in Fig~\ref{tr2}  illustrate that to attain a specified accuracy level, one can either increase the value of $\sigma_0$ or increase the size of the PML layer. For these examples, increasing the value of the PML strength above a certain threshhold does not improve the accuracy. In fact, in some cases the error gets worse with increasing $\sigma_0$. Results also show that higher order polynomial degrees require smaller PML sizes. For example for an $L^2$ accuracy of $10^{-6}$, the degree 3 splines require a PML width of 0.72 with $\sigma_0=10$, degree 5 splines require a PML width of 0.46 with $\sigma_0=12$, and the optimal PML width is 0.3 with $\sigma_0= 20$ when the degree of the splines is 8. In all examples, one needs a larger PML size in order to attain higher accuracy.

\section{Numerical Results}
In this section, we present numerical results that demonstrate the performance of bivariate spline functions 
in solving 
the exterior Helmholtz problem (\ref{exteriorP}). Our discretization seeks to approximate the solution of the 
PML problem (\ref{eqn:trpml}). 

 For comparison with the PML solution, we also test the performance of the first order absorbing boundary condition, 
that is the solution of 
\begin{eqnarray}
-\Delta u - k^2 u&=&f\;\hbox{ in } \Omega,\nonumber \\
u &=& g\; \hbox{ on }\;\Gamma, \nonumber \\
\dfrac{\partial u}{\partial n}-iku &=& 0\;\hbox{ on }\;\Gamma_{\!_{PML}}.
\end{eqnarray}
\label{eqn:somf}
The Robin boundary condition on $\Gamma_{\!_{PML}}$ is a low order truncating boundary which can be added easily to a PML code. In 
fact, we approximate the Sommerfeld radiation condition by imposing the Robin boundary condition on the exterior boundary of the 
PML problem (\ref{eqn:trpml}) and then set the PML parameter $\sigma_0=0$.

We choose the PML functions for $j=1,2$ as 
\[
  \sigma_j(x_j) = \left\{\begin{array}{@{}l@{\quad}l}
      0 & \mbox{if $|x_j|<a_j$,} \\[\jot]
      \sigma_0\left(\dfrac{|x_j|- a_j}{b_j-a_j}\right)^n & \mbox{if $a_j\leq |x_j|\leq b_j$} 
    \end{array}\right.
    \label{eqn:pmlf}
\]
where $[-a_1,a_1]\times [-a_2,a_2]$ is the inner PML boundary, and $[-b_1,b_1]\times [-b_2,b_2]$ is the outer PML boundary. In all 
numerical experiments , we set $n=4$. Other choices of $n=0,1,2,3,5,\cdots$ or even unbounded functions (see 
\cite{BHPR04, BHPR08}) can be made, but preliminary numerical results suggest that the fourth order PML 
function considered here 
results a better accuracy. 
The errors are measured by the relative $L^2$ and $H^1$ norms on the triangulation $\Delta$ defined by
$$ \mbox{L2error} = \dfrac{\|s_{\!_\triangle}-u\|_{L^2(\Delta)}}{\|u\|_{L^2(\Delta)}},\;\;\;\mbox{H1error}=\dfrac{\|\nabla( 
s_{\!_\triangle}- u)\|_{L^2(\Delta)}}{\|\nabla u\|_{L^2(\Delta)}},$$ 
where $s_{\!_\triangle}$ is the computed spline solution and $u$ the exact solution. In our computation, 
we actually use the discrete $\ell_2$ errors based on $500\times 500$ equally-spaced points 
over $[-b_1,b_1]\times [-b_2, b_2]$ exterior of $D$ to compute these relative errors above. 

\subsection*{Example 1}
\label{exactscat}

Our first example is a detailed investigation of scattering of a plane wave from a disk $D_2$ of radius $a=2$ 
centered at the  origin. We impose the  PML layer inside the rectangular annulus $[-5,5]^2\setminus(-3,3)^2$. 
The computational region of interest 
for the problem is the annular region outside the circle $(-3,3)\setminus\overline{D_2}$. It is well known that 
if the plane wave is propagating along the positive $x$-axis $\bfd = (1,0)$, then the solution of the 
scattering problem can be expressed as a series of Hankel functions in polar coordinates
\begin{eqnarray}
u(r,\theta) &=& -\left[\dfrac{J_0(ka)}{H_0^{(1)}(ka)}H_0^{(1)}(kr) + 2\sum_{m=1}^\infty i^m\dfrac{J_m(ka)}{H_m^{(1)}(ka)}H_m^{(1)}(kr)\cos(m\theta)\right],
\end{eqnarray}
where $J_m(z)$ denotes the Bessel function of order $m$ and $H_m^{(1)}(z)$ is the Hankel function of the first 
kind and order $m$. In this example, a simple calculation verifies that $f=0$.

The incident field is a plane wave $\exp(ik\bfx\cdot\bfd)$, and we take Dirichlet data on the boundary of the 
scatterer as $$g(\bfx) = u(r,\theta)|_{\Gamma}.$$
Our choice of boundary data eliminates errors due to the approximation of the disk by regular polygons, since 
we are only interested in errors due to discretization and due to the PML. 
The minimal tolerance for the iterative method was set at $\varepsilon = 10^{-10}$. 

\begin{figure}[htbp]
  \centering
  \label{fig:circlemesh}
\includegraphics[scale=0.2]{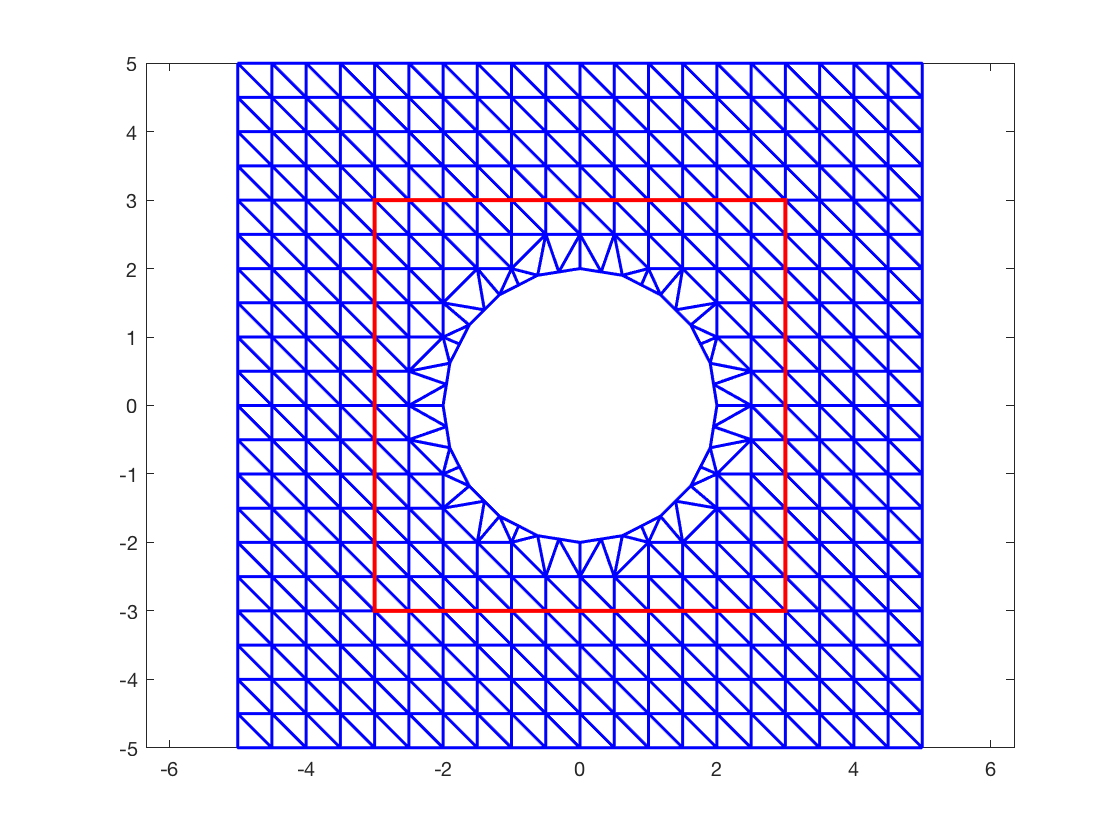}
\includegraphics[scale=0.55]{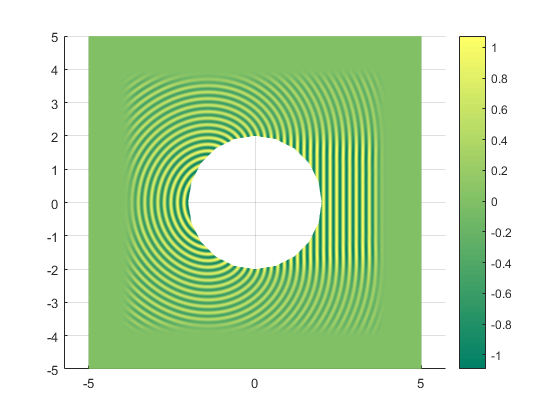}
  \caption{Meshes in Example 1 for $h=0.5$. Left: structured mesh aligned with inner PML boundary. Right: PML solution of the exterior Helmholtz equation with a wavenumber of $k=32$.}
  \label{fig:meshcircle}
\end{figure}
\begin{table}[h]\small
\begin{center}
 \begin{tabular}{|c| c| c| c| c|} 
 \hline
 $h$ & \# dofs & PML $H^1$-error  & PML $L^2$-error\\ [0.5ex] 
 \hline
 1.000 &3780 & 5.206370e-01 &4.820429e-01  \\ 
 
 0.500 & 15120&3.428191e-02& 1.821518e-02\\
 
  0.333 & 33264 &5.605999e-03&1.752080e-03\\  

  0.250 &59220& 1.292328e-03& 2.777040e-04\\  

  0.200&92736&4.314135e-04 &6.456502e-05\\   
  [1ex] 
 \hline 
\end{tabular}
\caption {Example 1.1: scattering from a disk, structured mesh PML solution with $\sigma_0=13$: $H^1$ and $L^2$ 
errors  using fifth order splines $S^1_5(\Delta)$ for $h=1/\ell$, $\ell=1,2,3,4$  for $k=8$. The computational 
domain is the annular region outside a disk of radius 2, $[-3,3]^2\backslash \overline{D_2}$ and the PML layer 
is $[-5,5]^2\backslash (-3,3)^2.$ } \label{tab:res1} 
\end{center} 
\end{table}

\begin{table}[h]\small
\begin{center}
 \begin{tabular}{|c|c|c|c||c|c|} 
 \hline
 $d$ & \# dofs & PML $H^1$-error &ABC1 $H^1$ error & PML $L^2$-error&ABC1 $L^2$ error\\ [0.5ex] 
 \hline
 5 & 15120 &3.428191e-02&1.110241e-01 &1.821518e-02&1.133321e-01\\ 
 
 6 & 20160& 4.833353e-03&9.584999e-02 &1.555229e-03&1.006903e-01\\
 
7 & 25920& 9.907632e-04&8.816546e-02 & 2.898009e-04&9.254443e-02\\
 
 8 &32400  &1.196707e-04&8.274271e-02 &2.746106e-05&8.680763e-02 \\
 
  9 & 39600 &2.111239e-05 &7.864768e-02 &4.794616e-06&8.247796e-02\\  

 10 &47520 &2.013801e-06& 7.541922e-02&5.266287e-07&7.906767e-02 \\  
  [1ex] 
 \hline 
\end{tabular}
\caption{Example 1.2: scattering from a disk. Comparison of the PML and first order ABC on a structured mesh. $H^1$ and $L^2$ errors of spline solutions of various degrees $S^1_d(\Delta)$ for $d=5,6,7,8,9,10$, for wavenumber $k=8$ and meshwidth $h=0.5.$  The computational region of interest is the annular region outside a disk of radius 2, $[-3,3]^2\backslash \overline{D_2}$ and the PML layer is $[-5,5]^2\backslash (-3,3)^2$} \label{tab:res2} 
\end{center} 
\end{table}

\noindent
EXAMPLE 1.1: \emph{The $h$-version}\\
We first investigate convergence by refining a structured mesh uniformly so that $h = 1, 1/2, 1/3, 1/4, 1/5$ while keeping the wavenumber fixed at $k=8$. Solutions are sought in the space $S^1_5(\Delta)$. Results are shown in Table \ref{tab:res1}. Both the $L^2$ and $H^1$ errors reduce by an order of magnitude with each refinement level. However the matrix size of the problem increases drastically with each uniform refinement, requiring $92736$ degrees of freedom for $h=0.2$ to reach an accuracy level of $4.3\times 10^{-2}\%$ error in the $H^1$ norm. \\

\noindent
EXAMPLE 1.2: \emph{The $p$-version}\\
Our next experiment  investigates convergence when the polynomial degree $d$ is raised when $k=8$  and $h=0.5$ are fixed. 
Solutions are sought in the spaces $S^1_d(\Delta)$ for $d=5,6,7,8,9,10$. Results are presented in Table \ref{tab:res2}. 
In Table \ref{tab:res2}, we compare the PML solution with that of the first order absorbing boundary condition 
(see problem \ref{eqn:somf}). As expected, the PML drastically outperforms the first order ABC. The error from the first order ABC 
stagnates at $7\%$ even using degree $10$ splines. This suggests that the dominant source of error is the ABC rather than the 
discretization, and the importance of accurate truncation of the exterior Helmholtz problem. Table \ref{tab:res2} also shows that 
raising the degree of the splines rather than refining the mesh results in a significant reduction in the degrees of freedoms 
necessary to get a specified accuracy level. For example when $d=10$, the $H^1$ error is about $2\times 10^{-4}\%$ 
requiring $47520$ degrees of freedom. 
\begin{table}[h]\small
\begin{center}
 \begin{tabular}{|c| c |c |c|} 
 \hline
 $d$ & \# dofs & PML $H^1$-error& PML $L^2$-error\\ [0.5ex] 
 \hline
 5&829,416 &1.007229e+00&1.004805e+00 \\ 
 
 6 &1,105,888&6.792040e-01&6.770829e-01 \\
 
 7 &1,421,856 &8.881512e-02&8.472434e-02\\
 
 8 &1,777,320&9.921788e-03&6.929828e-03\\
 
  9 & 2,172,881&2.048858e-03&8.027105e-04\\  

  10 &2,606,736&4.099069e-04&1.26832e-04\\  
  [1ex] 
 \hline 
\end{tabular}
\caption {Example 1.3: scattering from a disk, for a large wavenumber $k=100$,  meshwidth $h=1/15$ on the domain in Figure \ref{fig:meshcircle}. PML solution with $\sigma_0=13$: $H^1$ and $L^2$ errors  using  $S^1_d(\Delta)$ for $d=5,6,7,8,9,10.$} \label{tab:k100} 
\end{center} 
\end{table}

\noindent
EXAMPLE 1.3: \emph{High Frequency Scattering}\\
In the next experiment, the wavenumber is set at $k=100$, and the mesh is fixed at $h=1/15$ on the mesh in Figure 
\ref{fig:meshcircle}. This is a high wavenumber, considering the characteristic length of the domain $kL\sim 800$. 
In Table \ref{tab:k100}, we investigate convergence of spline solutions to the Helmholtz scattering problem with $k=100$ in the 
spaces $S^1_d(\Delta)$ for polynomial degrees $d=5,6,7,8,9,10$. Results show good convergence at high wavenumber if $d$ is 
chosen large enough.
\begin{figure}[htbp]\small
  \centering
  \label{fig:largemesh}
\includegraphics[scale=0.4]{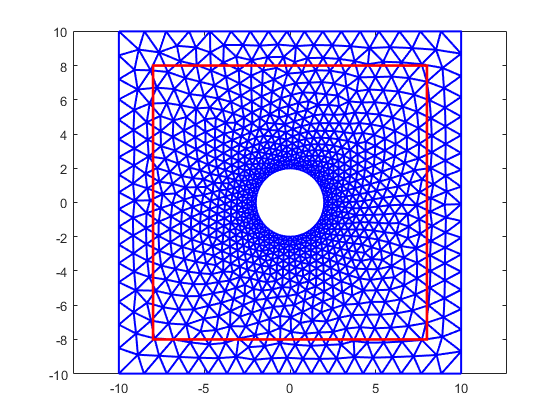}
\includegraphics[scale=0.15]{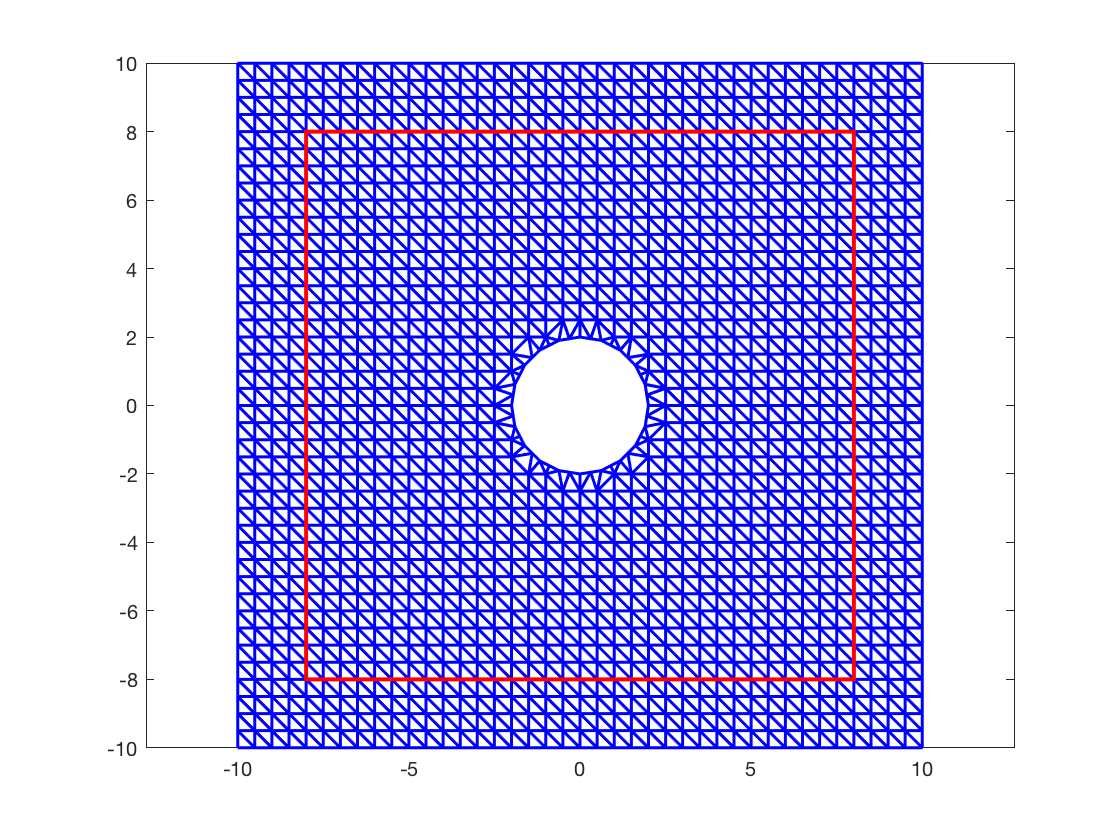}
\includegraphics[scale=0.15]{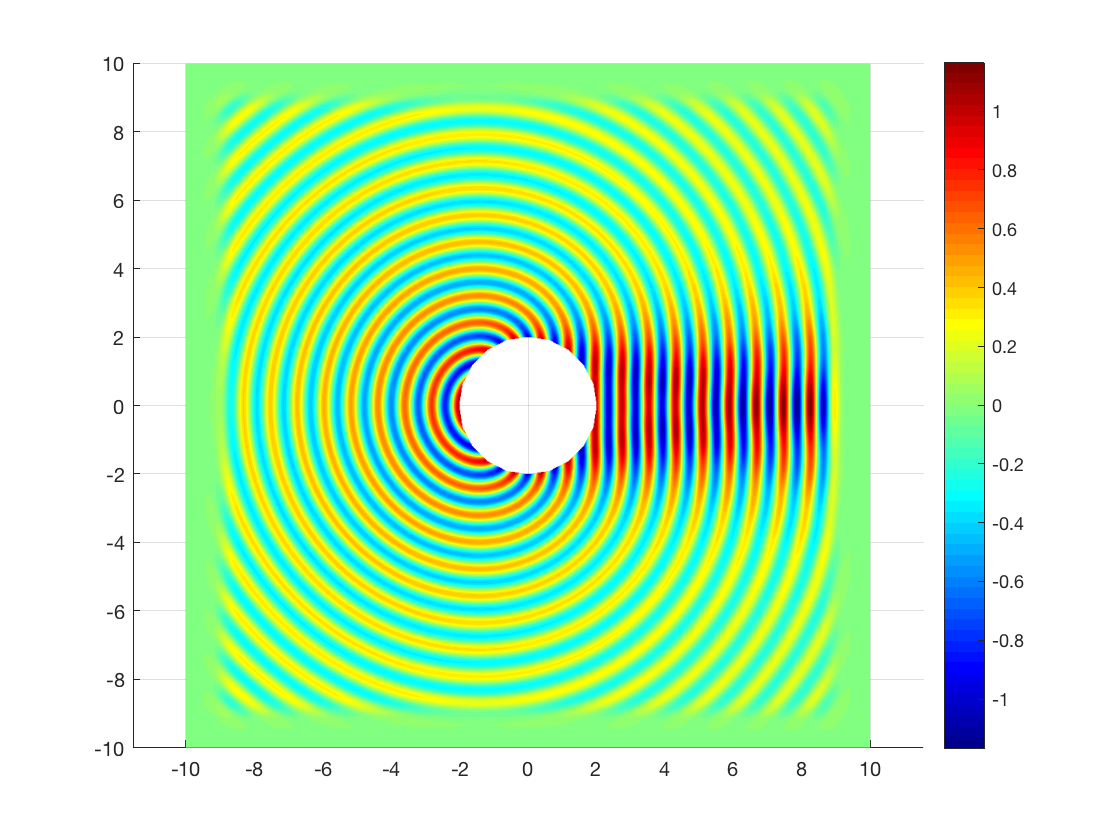}
\includegraphics[scale=0.4]{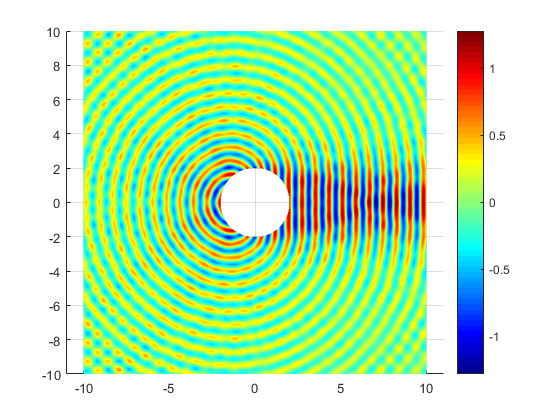}
  \caption{Top left: non-uniform mesh. Top right: uniform mesh. Bottom left: PML solution in $S^1_8(\Delta)$ for wavenumber $k=8$. 
Bottom right: First order ABC solution with a Robin boundary condition on the exterior boundary with wavenumber $k=8$. The effect 
of spurious reflections can be observed in the ABC solution. }
  \label{fig:meshlarge}
\end{figure}
\begin{table}[h]\small
\begin{center}
 \begin{tabular}{|c| c |c |c|} 
 \hline
 $d$ & \# dofs & PML $H^1$-error& PML $L^2$-error\\ [0.5ex] 
 \hline
 5&65520 &5.631785e-02  & 4.420372e-02 \\ 
 
 6 &87360 &6.686595e-03 &2.336916e-03  \\
 
 7 &112320 &1.214879e-03  &3.377661e-04\\
 
 8 &140400 &1.849320e-04 &4.058609e-05 \\
 
  9 & 171600 &2.635489e-05&5.472983e-06 \\  

  10 & 205920 &3.483980e-06&7.143510e-07\\  
  [1ex] 
 \hline 
\end{tabular}
\caption {convergence for scattering from a disk, uniform mesh with $h=0.5$. PML solution with $\sigma_0=13$: $H^1$ and $L^2$ 
errors  using splines $S^1_d(\Delta)$ for $d=5,6,7,8,9,10$, wavenumber $k=8$. The computational domain is the annular region 
$(-8,8)^2\backslash [-2,2]^2$ and the PML layer is $[-10,10]^2\backslash (-8,8)^2.$ }
 \label{tab:kplarge} 
\end{center} 
\end{table}
\begin{table}[h]\small
\begin{center}
 \begin{tabular}{|c| c |c |c|} 
 \hline
 $d$ & \# dofs & PML $H^1$-error& PML $L^2$-error\\ [0.5ex] 
 \hline
 5&42756 &2.075821e-01&1.953157e-01 \\ 
 
 6 &57008 &7.687323e-02 &7.205546e-02 \\
 
 7 &73296 &1.682735e-02&1.504267e-02\\
 
 8 &91620&3.660845e-03&2.922866e-03 \\
 
  9 & 111980 &8.114318e-04&5.587460e-04 \\  

  10 &134376 &2.211674e-04&1.501277e-04\\  
  [1ex] 
 \hline 
\end{tabular}
\caption {Example 1.4:$p$-version for a non-uniform mesh with $h=0.5$. PML solution with $\sigma_0=13$: $H^1$ and $L^2$ errors  
using  $S^1_d(\Delta)$ for $d=5,6,7,8,9,10$, wavenumber $k=8$. The computational region of interest is the annular region 
$(-8,8)^2\backslash\overline{D_2}$ and the PML layer is $[-10,10]^2\backslash (-8,8)^2.$ } \label{tab:kplarge1} 
\end{center} 
\end{table}

\noindent
EXAMPLE 1.4: \emph{A large domain}\\
In Tables \ref{tab:kplarge} and \ref{tab:kplarge1}, the domain of the scattering problem is increased to the annulus 
$[-10,10]\setminus\overline{D_2}$, and the PML layer is  set at $[-10,10]\setminus(-8,8)$, while the wavenumber is fixed at $k=8$. 
We test two cases of a structured mesh aligned with the PML boundaries and an unstructured mesh that cuts through the inner PML 
boundary. The unstructured mesh is graded so that the triangles in the PML layer are larger than triangles close to the scatterer, 
as shown in Figure \ref{fig:meshlarge}. We only  investigate convergence with respect to polynomial degree $d=5,6,7,8,9,10$ with a 
PML layer. Numerical results in Table \ref{tab:kplarge1} suggest good convergence even if the mesh is not aligned with the PML 
layer. There is a slight advantage to using unstructured meshes with larger triangles in the PML layer, with respect to the 
degrees of freedom required. An adaptive PML $hp$ procedure using error indicators is need to guarantee optimality with respect to 
mesh size and polynomial degrees.

\subsection*{Example 2}
This example is taken from Table 1 of \cite{KP10a}, 
where the authors use piecewise linear functions to solve a scattering 
problem with a PML layer. We are interested in using bivariate spline functions to solve this problem. This is an example of 
scattering of a spherical wave from a square $D = [-1,1]^2\subset\mathbb{R}^2$. The wavenumber is $k=2$, and the boundary 
condition on the surface of the scatterer is given by $g = e^{i\theta}H_1^{(1)}(kr)$. This implies that the exact solution is 
$u(r,\theta)=e^{i\theta} H_1^{(1)}(kr)$ in the exterior of $D$. Note that in this example $f=0$.

We choose the PML parameter $\sigma_0=13$. The region of computational interest is the annulus $(-3,3)^2\setminus [-1,1]^2$, and 
the PML layer is placed on $[-5,5]^2\setminus (-3,3)^2$. As in Table 1 of (\cite{KP10a}), we investigate convergence of the error 
in the $L^2$ and $H^1$ norms for different mesh sizes (see Figure \ref{fig:meshsquare}). \\

\noindent 
EXAMPLE 2.1: \emph{Continuous Piecewise Linears}\\
We first use continuous piecewise linears $S^0_1(\Delta)$ to replicate the results in Table 1 of (\cite{KP10a}) Results are 
shown in Table (\ref{tab:kp101}). 
We observe first order convergence in the $H^1$ norm and second order convergence in $L^2$ in agreement with the results 
in (\cite{KP10a}).
\begin{figure}[htbp]\small
  \centering
  \label{fig:squaremesh}\includegraphics[scale=0.4]{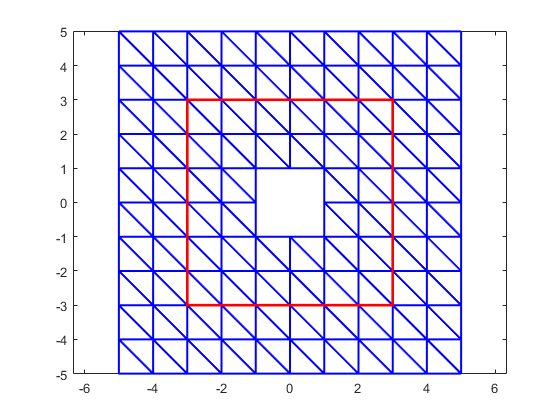}
\includegraphics[scale=0.4]{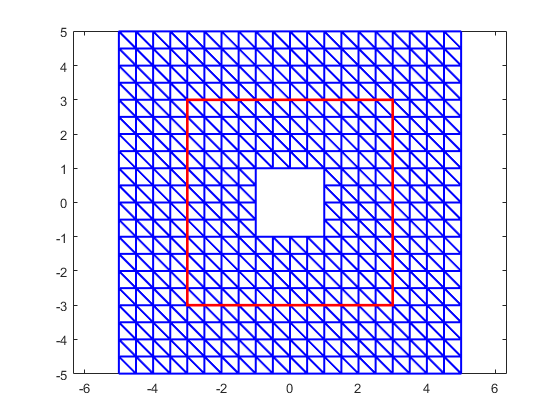}
  \caption{Meshes in Example 2 for $h=1, 0.5$. The PML layer is outside the red boundary.}
  \label{fig:meshsquare}
\end{figure}
\begin{table}[h]\small
\begin{center}
 \begin{tabular}{|c| c| c| c |c |c|} 
 \hline
 $h$ & \# dofs & PML $H^1$-error & ratio & PML $L^2$-error& ratio \\ [0.5ex] 
 \hline
 1 & 576 & 0.882 & 0.00 & 0.601 & 0.000 \\ 
 
 1/2 & 2304 & 0.421 & 2.10 & 0.234 & 2.572 \\
 
 1/4 & 9216 & 0.192 & 2.19 & 0.0676& 3.458\\
 
 1/8 & 36864 & 0.0938 & 2.05 & 0.0174 & 3.891\\
 
  1/16 & 147456 & 0.0463 & 2.03 &0.00437 & 3.97\\  
  [1ex] 
 \hline 
\end{tabular}
\caption {Example 2.1: The $h$-version. PML solution with $\sigma_0=13$: $H^1$ and $L^2$ errors  using linear splines 
$S^0_1(\Delta)$ for $h=1/2^\ell$, $\ell=0,1,2,3,4$ with exact solution $u(r,\theta) = e^{i\theta}H^{(1)}_1(kr)$ for $k=2$. The 
computational domain is the annular region $(-3,3)^2\backslash [-1,1]^2$ and the PML layer is $[-5,5]^2\backslash (-3,3)^2.$  
\label{tab:kp101} }
\end{center} 
\end{table}\\

\noindent
EXAMPLE 2.2: \emph{The $p$ version}\\
Next, we increase the degree of the spline spaces for a fixed mesh (we choose a course mesh $h=1$). The spline spaces 
$S^1_d(\Delta)\subset C^1(\Omega)$ are chosen with degree $d=5,6,7,8,9,10$. Results are shown in Table \ref{tab:kpd5}.
Compared with the results in Table \ref{tab:kp101}, we are able to obtain good accuracy with much fewer degrees of freedom on a 
relatively course mesh. 
\begin{table}[h]\small
\begin{center}
 \begin{tabular}{|c|c|c|c||c|c|} 
 \hline
 $d$ & \# dofs & PML $H^1$-error & ratio & PML $L^2$-error& ratio\\ [0.5ex] 
 \hline
 5 & 4032 & 3.739715e-03& - &1.368041e-03 &-\\ 
 
 6 & 5376 & 1.901169e-03&1.97&1.040349e-03&1.31\\
 
7 & 6912 & 4.698431e-04&4.05&2.910898e-04&3.57\\
 
 8 & 8640 & 1.488127e-04&3.16&1.162590e-04&2.50\\
 
  9 & 10560 & 4.826350e-05&3.08&1.599918e-05&7.27\\  

 10 & 12672 & 4.742504e-05&1.02&4.377348e-05&0.37\\  
  [1ex] 
 \hline 
\end{tabular}
\caption{Example 2.2: The $p$-version. PML solution with $\sigma_0=13$: $H^1$ and $L^2$ errors of spline solutions of various 
degrees $S^1_d(\Delta)$ for $d=5,6,7,8,9,10$, with exact solution $u(r,\theta) = e^{i\theta}H^{(1)}_1(kr)$ for wavenumber $k=2$ 
and mesh width $h=1.$  The computational domain is the annular region $(-3,3)^2\backslash [-1,1]^2$ and the PML layer is 
$[-5,5]^2\backslash (-3,3)^2$.  \label{tab:kpd5}}  
\end{center} 
\end{table}

\subsection*{Example 3}
EXAMPLE 3.1: \emph{Helmholtz equation in a variable medium}\\
In this section we illustrate the performance of bivariate splines for the variable media Helmholtz equation with 
a contrast function $b(\bfx)$ whose support is contained in $\Omega_F$.
\begin{eqnarray}
&&-\Delta u(\bfx)-k^2(1-b(\bfx))u(\bfx)= k^2b(\bfx)f(\bfx)\;\;\mbox{in}\;\;\Omega\\
&& u(\bfx) = g(\bfx)\;\;\mbox{on}\;\;\Gamma
\end{eqnarray}
satisfying the Sommerfeld radiation condition at infinity. This is similar to Example 4.1 of \cite{LG16} where the source function 
is a plane wave. We use a point source incident field. The domain $\Omega=[-3,3]^2\backslash \overline{B}_{0.25}$ and 
$\Omega_\PML=[-3,3]^2\backslash (-2,2)^2$ where $B_{0.25}$ is a disk of radius 0.25 centered at the origin. We choose the contrast 
function as $b(\bfx)=0.5\;\mbox{erfc}(5(|\bfx|^2-1))$ where the error function $\mbox{erfc}(r)$ is defined by 
$$\mbox{erfc}(r)=\dfrac{2}{\sqrt{\pi}}\int_r^\infty e^{-t^2}\;dt.$$
The function $b(\bfx)$ decays to zero quickly outside a bounded set (see Fig \ref{fig:bmp}). We take the data function as 
$f(\bfx)=H_0^{(1)}(k|\bfx|)$, and choose Dirichlet data on $\Gamma$ as $g(\bfx)=u(\bfx)|_{\Gamma}$. The exact solution is then 
$u(\bfx)=H_0^{(1)}(k|\bfx|)$.

\begin{table}[h]\small
\begin{center}
 \begin{tabular}{|c| c |c |c|} 
 \hline
 $d$ & \# dofs & PML $H^1$-error& PML $L^2$-error\\ [0.5ex] 
 \hline
 5&49,770 &2.521487e-03&1.703128e-03\\ 
 
 6 &66,360 &5.461371e-04&4.989304e-04 \\
 
 7 &85,320&1.630330e-04&1.677855e-04\\
 
 8 &106,650&6.051181e-05& 6.425829e-05 \\
 
  9 & 130,350 &1.688453e-05& 1.756126e-05 \\  

  10 &156,420 &7.292140e-06&7.701371e-06\\  
  [1ex] 
 \hline 
\end{tabular}
\caption {Example 3: Helmholtz equation in a variable medium. PML solution with $\sigma_0=10$: $H^1$ and $L^2$ errors of spline 
solutions of various degrees $S^1_d(\triangle)$ for $d=5,6,7,8,9,10$. The exact solution is $u(\bfx)=H_0^{(1)}(k|\bfx|)$ with 
$k=15$. The domain is shown in Figure \ref{fig:bmp}.  \label{tab:bmp} } 
\end{center} 
\end{table}

\begin{figure}[t!]\small
    \centering
    \begin{subfigure}[t]{0.5\textwidth}
        \centering
        \includegraphics[height=2in]{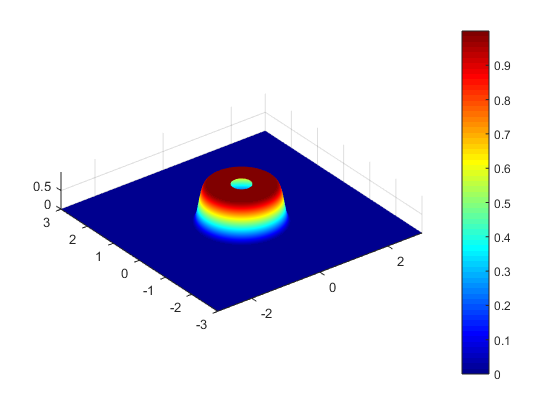}
        \caption{Bump function $b(\bfx)=0.5\mbox{erfc}(5(|\bfx|^2-1))$.}
    \end{subfigure}%
    ~
 \begin{subfigure}[t]{0.5\textwidth}
        \centering
        \includegraphics[height=2in]{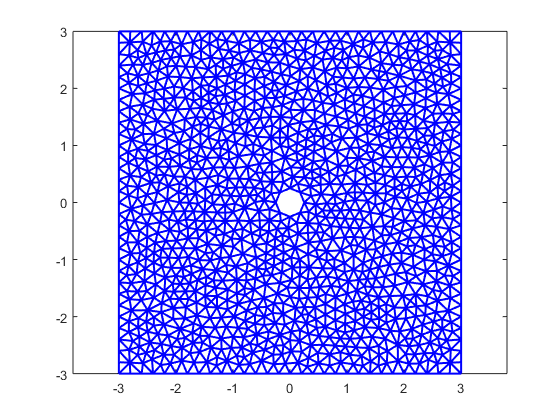}
        \caption{Mesh}
    \end{subfigure}

    \caption{The contrast function and mesh in Example 3.}
\end{figure}\label{fig:bmp}

\begin{figure}[t!]
\begin{tabular}{cc}
\includegraphics[height=2in]{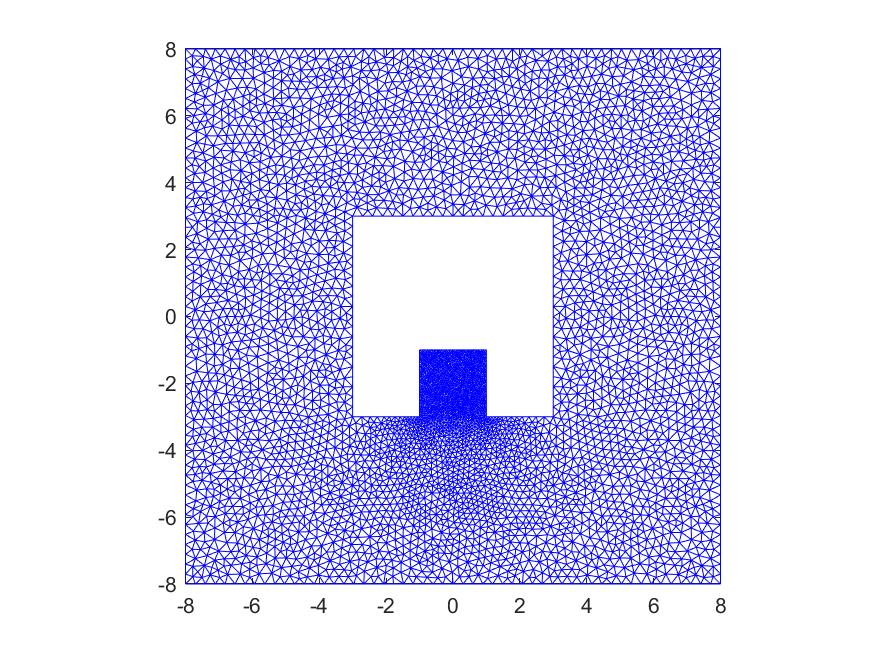} & \includegraphics[height=2in]{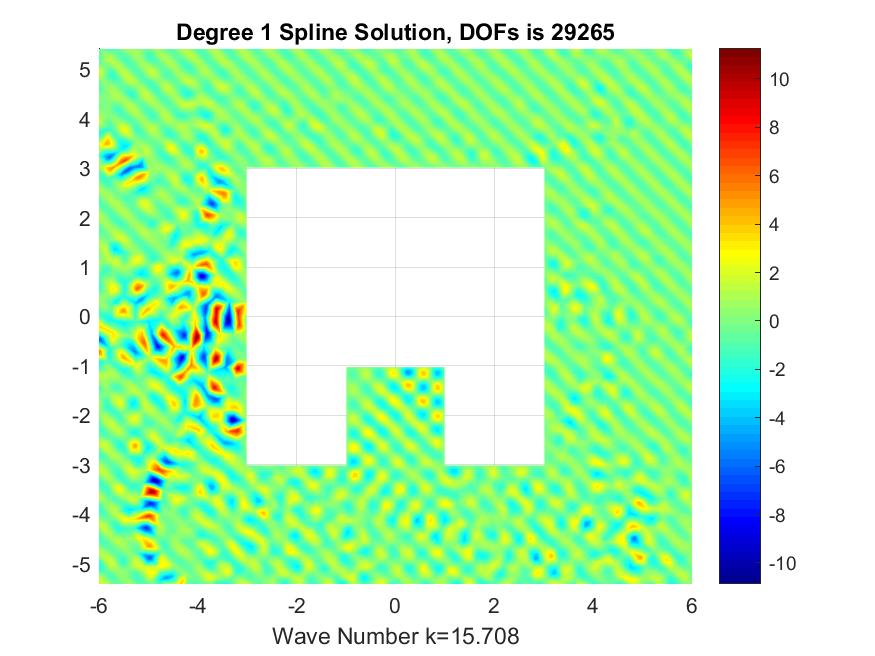} \cr
\includegraphics[height=2in]{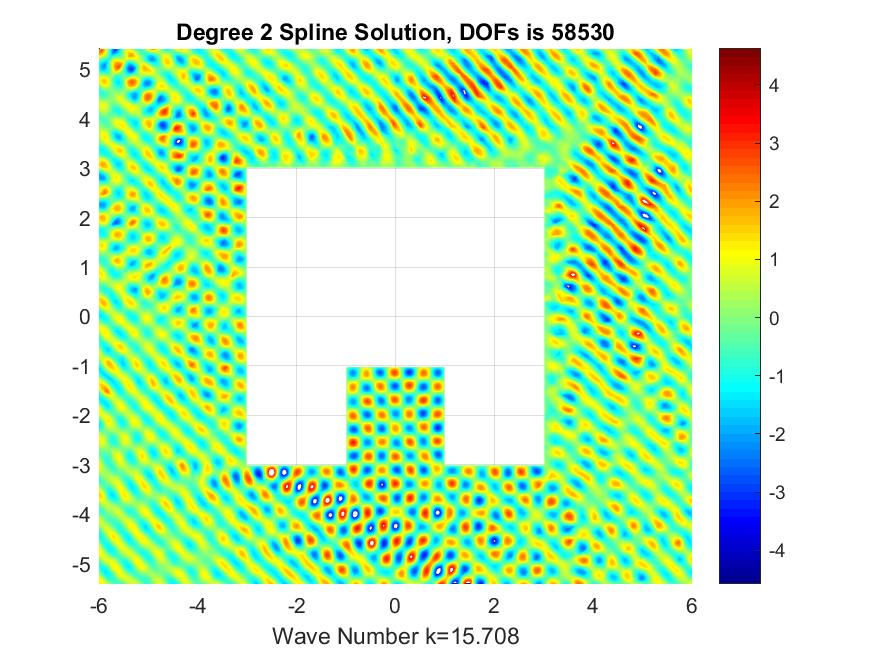} & \includegraphics[height=2in]{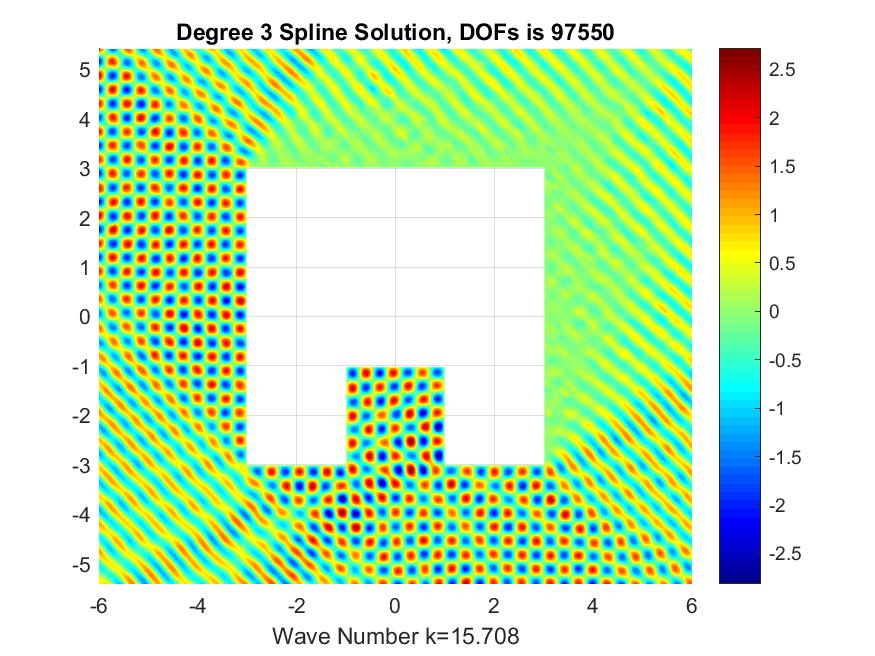} \cr 
\includegraphics[height=2in]{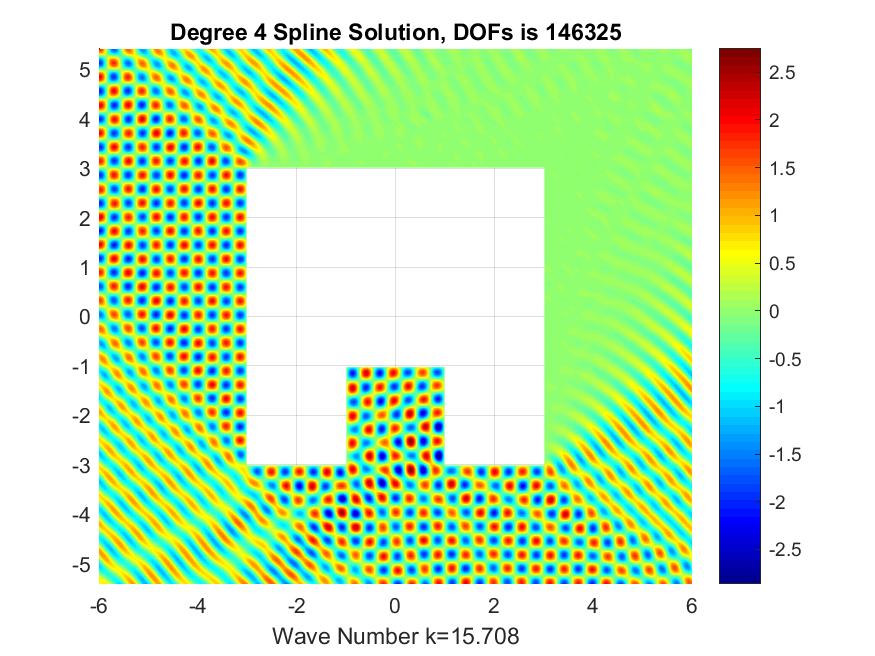} & \includegraphics[height=2in]{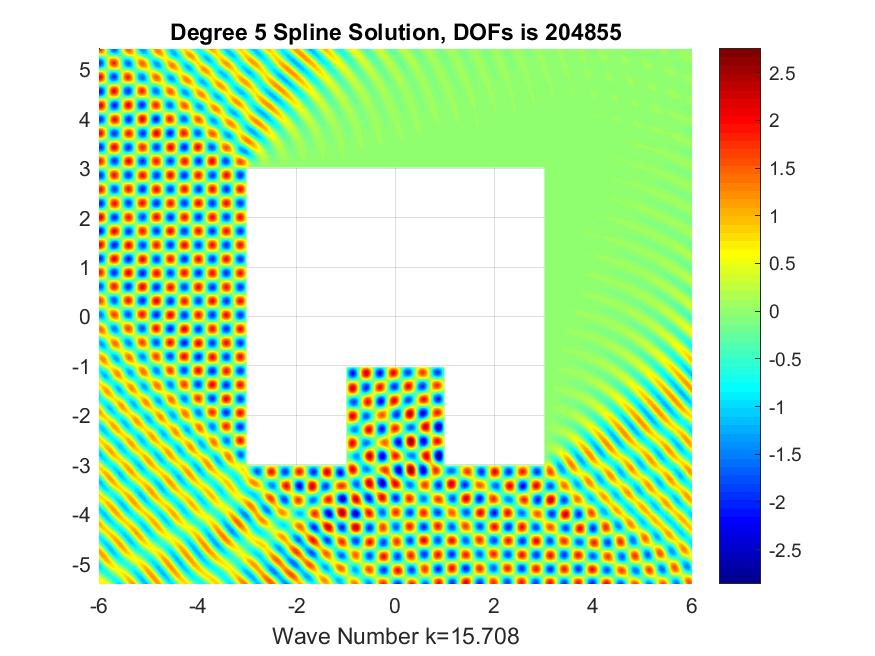} \cr
\end{tabular}
\caption{The truncated sound soft scattering problem using spline functions of degree $d=1, 2, \cdots, 5$. When $d\ge 3$, 
the resulting spline solution well resolves total  wave (incident plus scattered).
\label{tr1}}
\end{figure}

\begin{figure}[htpb]
\begin{tabular}{cc}
\includegraphics[height=2in]{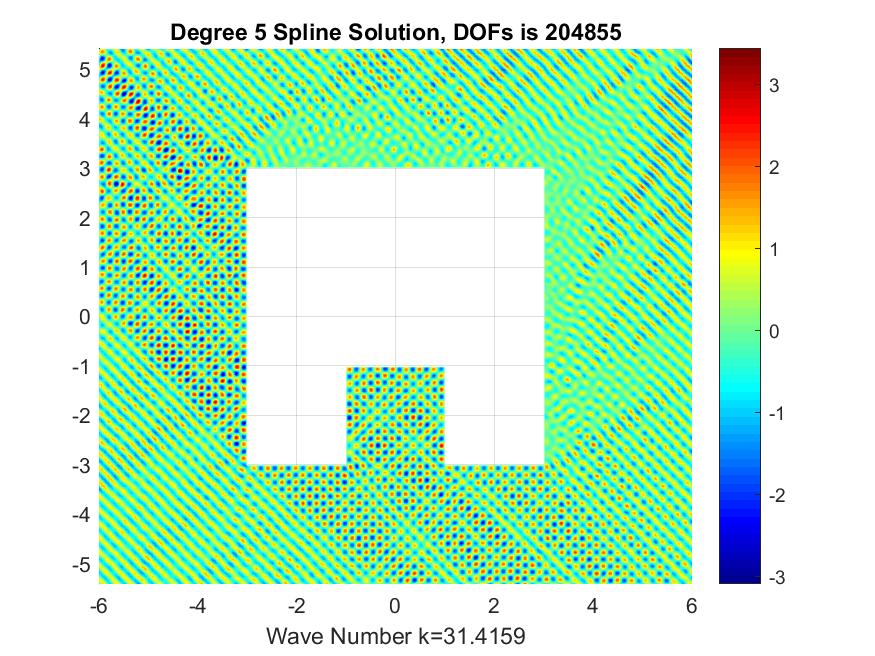} & \includegraphics[height=2in]{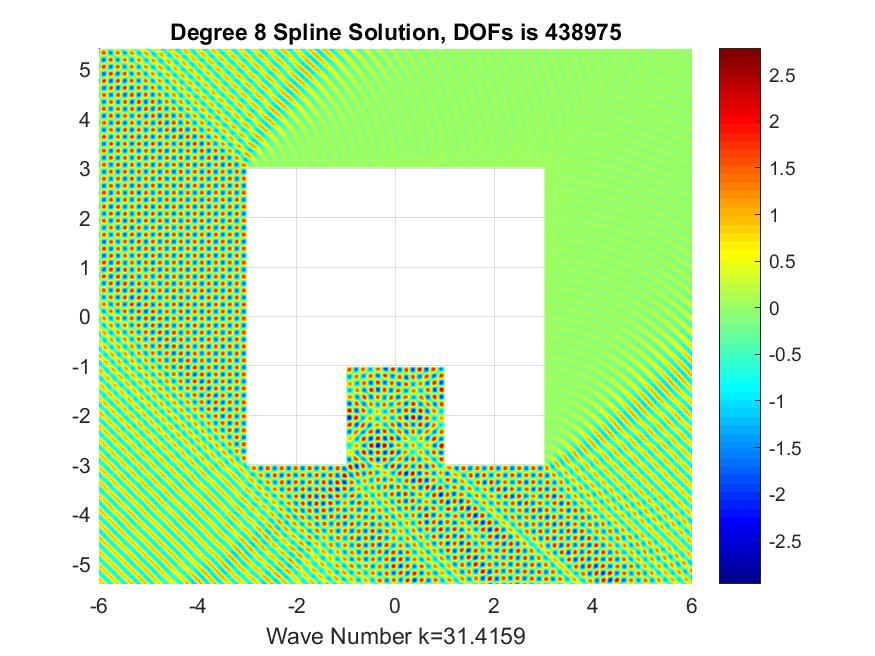} \cr
\end{tabular}
\caption{The truncated sound soft scattering problem using spline functions of degree $d=5$ and $d=8$. When $k=10\pi$, 
the spline solution of degree 5 does not produce an accurate approximation, but the spline solution of degree 8 does.  
\label{tr2}}
\end{figure}
\subsection*{Example 4} 
This example is taken from \cite{GGS15} where the authors investigate the use of a pre-conditioner to solve the truncated sound 
soft scattering problem based on the shifted-Laplace equation. See \cite{GSV17} for another preconditioning approach.  
The domain and mesh are depicted in Figure~\ref{tr1} (top left panel). The outer boundary is a square of
size $12\times 12$. The obstacle is a square of size $6\times 6$ placed symmetrically inside, 
with a $2\times 2$ square removed from the bottom side. 
We solve the truncated sound soft scattering
problem  with an incident plane wave coming from the bottom left corner of the
graph, with the direction of propagation at 45 degrees with the positive $x$-axis.
The scatterer in this example is trapping, since there exist closed paths of rays in its exterior.  
The researchers in \cite{GGS15} experimented the GMRES and used a preconditioner to find accurate solutions for
various wave numbers $k$. 
From spline solutions of various degrees shown in Figure~\ref{tr1}, 
we can see that when $d=1$ and $d=2$, the solutions are not accurate enough in the sense the waves are considerably less 
well resolved. 
When the spline degree is raised to $d\ge 3$, the solutions get better (see Figure~\ref{tr1}), that is,  
the resulting total field (incident plus scattered) is well resolved. Even  using bivariate spline functions 
of degree 5, the solution is generated within 200 seconds a laptop computer. Furthermore, for wave number $k=10\pi$, the 
degree 5 spline solution is not well-resolved as shown in Figure~\ref{tr2}. However, when we use spline functions 
of degree 8, we are able to find a very accurate approximation. See the right graph of Figure~\ref{tr2}.  
For large wave number, $k=500$ or larger, see the bivariate spline solution in \cite{LM18}. 

\subsection*{Example 5}
In our final numerical experiment, we investigate a ``real life" example of the scattering of a point source located at one focus 
of an elliptic amphitheater (see Figure \ref{fig:ellipse}). We impose Dirichlet boundary conditions on the boundary of the 
elliptic walls associated with the point source and the fundamental solution of the Helmholtz equation $$ u(\bfx)|_{\!_{\Gamma}} = 
-\frac{i}{4}H_0^{(1)}(k|\bfx-\bfx_0|), $$
where $\bfx_0$ is the left focus of the inner ellipse. The inner and outer PML boundaries are the rectangles $[-2,2]^2$ and 
$[-2.2,2.2]^2$. In architectural acoustics, it is well known that a point source at the left focus leads to a maximum sound 
pressure at the right focus, in agreement with our results in Figure \ref{fig:ellipse}.
\begin{figure}[t!]\small
    \centering
    \begin{subfigure}[t]{0.5\textwidth}
        \centering
        \includegraphics[height=2in]{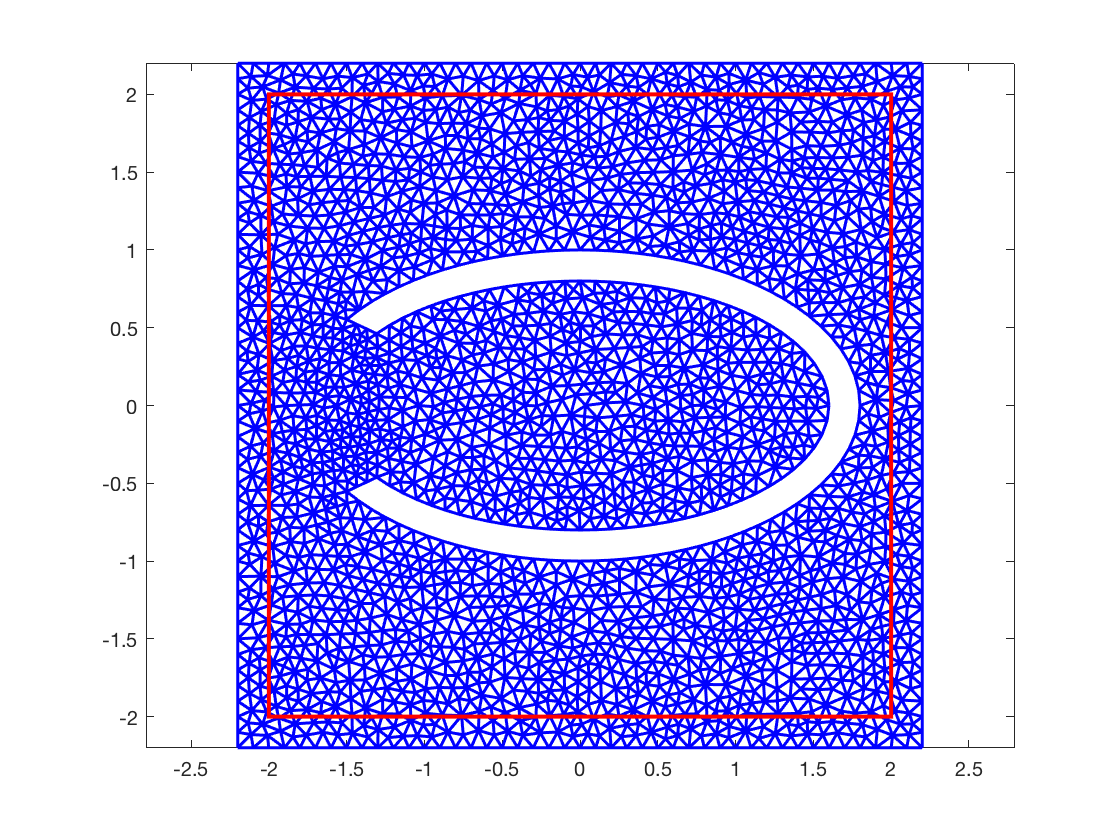}
        \caption{Mesh for elliptic section with a PML layer.}
    \end{subfigure}%
    ~ 
    \begin{subfigure}[t]{0.5\textwidth}
        \centering
        \includegraphics[height=2in]{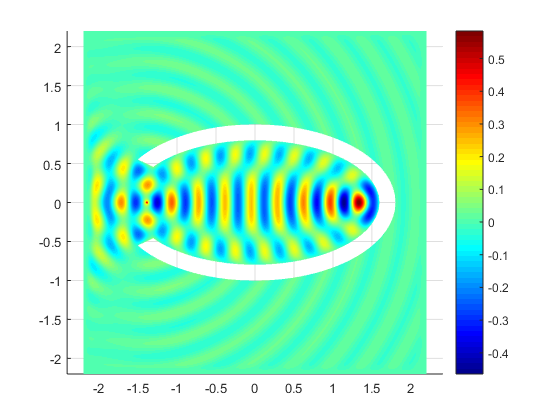}
        \caption{real part of the PML solution, wavenumber $k=6\pi$}
    \end{subfigure}
  ~
 \begin{subfigure}[t]{0.5\textwidth}
        \centering
        \includegraphics[height=2in]{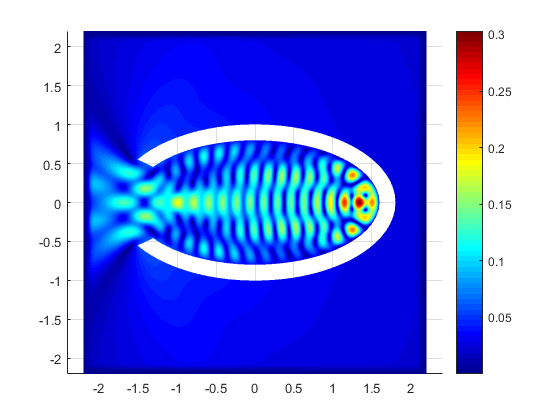}
        \caption{modulus of the PML solution, wavenumber $k=6\pi$}
    \end{subfigure}
\caption{Example 4: PML solution of acoustic wave propagation in a horizontal cross section of an elliptic amphitheater. A source 
at one focus leads to maximum sound pressure at the other focus. \label{fig:ellipse}}
\end{figure}

\section{Conclusion}
We have shown how to approximate the solution of the exterior Helmholtz equation by using bivariate splines 
with a PML layer. Numerical results suggest good accuracy when using high order splines. Extension of our
study to the Helmholtz equation in heterogeneous and anisotropic media will be reported elsewhere. Extension of
the bivariate spline method to the 3D setting will be investigated in the near future. 

\end{document}